\documentclass[a4paper,draft]{amsproc}
\usepackage{amssymb}
\usepackage{amscd} %% Package for commutative diagrams
\usepackage{snapshot}
\usepackage[dvips]{graphicx}
\usepackage{color}
\usepackage{subfigure}
\usepackage{rotating}
\usepackage{amsmath}
\usepackage{amsthm}
\usepackage{amsfonts}
\theoremstyle{plain}
 \newtheorem{theorem}{Theorem}[section]
 
 \newtheorem{lemma}{Lemma}[section]
 \newtheorem{corollary}{Corollary}[section]

 \newtheorem{definition}{Definition}[section]
\theoremstyle{remark}
 
 \newtheorem{notation}{Notation}[section]
 \numberwithin{equation}{section}

\renewcommand{\leq}{\leqslant}
\renewcommand{\geq}{\geqslant}

\title[]{Fuzzy Riesz subspaces, fuzzy ideals, fuzzy bands and fuzzy band projections}


\author[Hong]{\bfseries Liang Hong}

\address{
Department of Mathematics \\ % \hfill (Received 00 00 2010)\\
Robert Morris University   \\ %\hfill (Revised  00 00 2010)\\
Moon, PA 15108, USA}
\email{hong@rmu.edu}

\begin{document}

\vspace{18mm}
\setcounter{page}{1}
\thispagestyle{empty}

\begin{abstract}
Fuzzy ordered linear spaces, Riesz spaces, fuzzy Archimedean spaces and $\sigma$-complete fuzzy Riesz spaces were defined and studied in several works. Following the efforts along this line, we define fuzzy Riesz subspaces, fuzzy ideals, fuzzy bands and fuzzy band projections and establish their fundamental properties.
\end{abstract}

\maketitle

\section{Introduction}  %% Please avoid complicated formulas in titles
The theory of fuzzy mathematics was initiated in \cite{Zadeh1} and the notion of fuzzy order relation
was first defined in \cite{Zadeh2}. Later \cite{Venugopalan1} developed a systematic framework of fuzzy ordered sets paralleling that of classical partially ordered sets. This naturally led to the studies on fuzzy Riesz spaces in \cite{BI0}, fuzzy ordered linear spaces in \cite{BI1}, fuzzy Archimedean spaces in \cite{BI2} and $\sigma$-complete fuzzy Riesz spaces in \cite{Beg1}. \cite{Beg3} provides a good review of the key results in this direction. The purpose of this paper is to define and study fuzzy Riesz subspaces, fuzzy ideals fuzzy bands and fuzzy projection bands.

We fix some notations for our presentation. Unless otherwise stated, $N$ denotes the set of natural numbers; $R$ denotes the set of real numbers; $R^+$ denotes the set of nonnegative real numbers; Greek letters $\alpha, \beta, ...$ denote either indices or real numbers; the symbols $\leq$ and $>$ are used with respect to the usual order on $R$; all functions are assumed to be real-valued.

The remainder of the paper is organized as follows. Section 2 provide readers with some preliminaries; most material in this section can be found in the papers cited above; we give a few counterexamples to complement the existing literature; for a detailed treatment of fuzzy set theory, we refer to \cite{WRK} and \cite{Zimmermann}; for a comprehensive treatment of the classical theory of Riesz spaces, we refer to \cite{LZ}. Section 3 defines fuzzy ideals and studies their basic properties. Section 4 defines fuzzy bands and gives several important properties. Section 5 is devoted to the investigation of fuzzy projection bands.

\section{Preliminaries}

\subsection{Fuzzy ordered sets and fuzzy lattices}
\begin{definition}\emph{\ \cite{Zadeh1} }
\emph{Let $X$ be a space of points, with a generic element of $X$ denoted by $x$. A \emph{fuzzy set} $A$ on $X$ is a membership function $\mu_A: X\rightarrow [0, 1]$, with the value of $\mu_A(x)$ at $x$ representing the ``grade of membership'' of $x$ in $A$. The nearer the value $\mu_A(x)$ to unity, the higher the grade of membership of $x$ in $A$.}
\end{definition}
\noindent \textbf{Remark.} To distinguish a fuzzy set from an ordinary set, we call an ordinary set a \emph{crisp set}.

\begin{definition}\emph{\ \cite{Zadeh2} }
\emph{Let $X$ be a crisp set. A \emph{fuzzy order} on $X$ is a fuzzy subset of $X\times X$ whose membership function $\mu$ satisfies
\begin{enumerate}
  \item [(i)](reflexivity) $x\in X$ implies $\mu(x, x)=1$;
  \item [(ii)](antisymmetric) $x, y\in X$ and $\mu(x, y)+\mu(y, x)>1$ implies $x=y$;
  \item [(iii)](transitivity) $x, z\in X$ implies $\mu(x, z)\geq \vee_{y\in X}[\mu(x, y)\wedge \mu(y, z)]$, where $\vee$ and $\wedge$ denote supremum and infimum with respect to the usual order, respectively.
\end{enumerate}
A set with a fuzzy order defined on it is called a \emph{fuzzy ordered set} (or \emph{foset} for short.) }
\end{definition}

\begin{notation}\emph{\ \cite{Venugopalan1} }
\emph{Let $X$ be a foset and $x\in X$. $\uparrow x$ denotes the fuzzy set on $X$ defined by $(\uparrow x)(y)=\mu(x, y)$ for all $y\in X$. Likewise, $\downarrow x$ denotes the fuzzy set on $X$ defined by $(\downarrow x)(y)=\mu(y, x)$ for all $y\in X$. If $A$ is a crisp subset of $X$, $\uparrow A=\cup_{x\in A}(\uparrow x)$ and $\downarrow A=\cup_{x\in A}(\downarrow x)$.}
\end{notation}

\begin{definition}\emph{\ \cite{Venugopalan1} }
\emph{Let $A$ be a crisp subset of a foset $X$. The \emph{upper bound} $U(A)$ of $A$ is  the fuzzy set on $X$ defined as \begin{equation*}
U(A)(y)=\left\{
          \begin{array}{ll}
            0, & \hbox{if $(\uparrow x)(y)\leq 1/2$ for some $x\in A$;} \\
            \left(\cap_{x\in A} \uparrow x \right)(y), & \hbox{otherwise.}
          \end{array}
        \right.
\end{equation*}
Likewise, the \emph{lower bound} $L(A)$ of $A$ is the fuzzy set on $X$ defined as
\begin{equation*}
L(A)(y)=\left\{
          \begin{array}{ll}
            0, & \hbox{if $(\uparrow x)(y)\leq 1/2$ for some $x\in A$;} \\
            \left(\cap_{x\in A} \downarrow x \right)(y), & \hbox{otherwise.}
          \end{array}
        \right.
\end{equation*}
If $U(A)(x)>0$ for some $x\in X$, we write $x\in U(A)$; in this case we say $A$ is \emph{bounded above} and we call $x$ an \emph{upper bound} of $A$. Similarly, if $L(A)(x)>0$, we write $x\in L(A)$; in this case we say $A$ is \emph{bounded below} and we call $x$ a \emph{lower bound} of $A$. $A$ is said to be \emph{bounded} if it is both bounded above and bounded below. An element $z\in X$ is said to be a \emph{supremum} of $A$ if (i) $z\in U(A)$ and (ii) $y\in U(A)$ implies $y\in U(z)$. An element $z\in X$ is said to be a \emph{infimum} of $A$ if (i) $z\in L(A)$ and (ii) $y\in L(A)$ implies $y\in L(z)$. For a fuzzy subset $S$ of a foset $X$, $U(S)$ denotes $U(supp S)$, where $S=\{x\in X\mid \mu_S(x)>0\}$ is called the \emph{support} of $S$. Similarly, $L(S)$ denotes $L(supp S)$.}
\end{definition}

\begin{theorem}\emph{\ \cite{Venugopalan1} } \label{theorem2.1.1}
Let $A$ be a subset of a foset $X$. Then
\begin{enumerate}
  \item [(i)]$\inf A$, if it exists, is unique;
  \item [(ii)]$\sup A$, if it exists, is unique.
\end{enumerate}
\end{theorem}

\begin{notation}\emph{\ \cite{Venugopalan1} } \label{theorem2.1.2}
$x\vee y=\sup\{x, y\}$ and $x\wedge y=\inf\{x, y\}$.
\end{notation}

\begin{theorem}\emph{\ \cite{Venugopalan1} } \label{theorem2.1.3}
Let $X$ be a foset. Then the following identities hold, whenever the expressions referred to exist.
\begin{enumerate}
  \item [(i)](idempotent) $x\wedge x=x$ and $x\vee x=x$.
  \item [(ii)](commutative) $x\wedge y=y\wedge x$ and $x\wedge y=y\wedge x$.
  \item [(iii)](absorption) $x\wedge (x\vee y)=x\vee(x\wedge y)=x$.
  \item [(iv)]$\mu(x, y)>1/2$ if and only if $x\wedge y=x$ if and only if $x\vee y=y$.
\end{enumerate}
\end{theorem}

\begin{definition}\emph{\ \cite{Venugopalan1} }
\emph{A foset $X$ is called a \emph{fuzzy lattice} (or \emph{F-lattice} for short) if all finite subsets of $X$ have suprema and infima. A fuzzy lattice is said to be \emph{complete} if every subset of $X$ has a supremum and an infimum.}
\end{definition}

\subsection{Fuzzy Riesz spaces}
\begin{definition}\emph{\ \cite{BI1} }\label{definition2.3.1}
\emph{A real vector space $X$ is said to be a \emph{fuzzy ordered vector space} if $X$ is a foset and the fuzzy order on $X$ is compatible with the vector structure of $X$ in the sense that it satisfies the following two properties:
\begin{enumerate}
  \item [(i)]if $x, y\in X$ satisfies $\mu(x, y)>1/2$, then $\mu(x, y)\leq \mu(x+z, y+z)$ for all $z\in X$;
  \item [(ii)]if $x, y\in X$ satisfies $\mu(x, y)>1/2$, then $\mu(x, y)\leq \mu(\lambda x, \lambda y)$ for all $\lambda\in R^+$.
\end{enumerate}}
\end{definition}
\noindent \textbf{Remark.} It follows from the transitivity of $\mu$ and condition (i) that if $\mu(x_1, x_2)>1/2$ and $\mu(x_3, x_4)>1/2$, then $\mu(x_1+x_3, x_2+x_4)>1/2$.

\begin{definition}\ \emph{\cite{BI1}} \label{definition2.3.2}
\emph{Let $X$ be a fuzzy ordered vector space and $x\in X$. $x$ is said to be \emph{positive} if $\mu(0, x)>1/2$; $x$ is said to be \emph{negative} if $\mu(x, 0)>1/2$; $x$ is said to be \emph{nonnegative} if $x$ is not negative.}
\end{definition}

\begin{definition}\ \emph{\cite{BI2}}
\emph{Let $D$ be a subset of foset $X$.
\begin{enumerate}
  \item [(i)]$D$ is said to be \emph{directed to the right} if for every finite subset $E$ of $D$, $D\cap U(E)\neq \phi$.
  \item [(ii)]$D$ is said to be \emph{directed to the left} if for every finite subset $E$ of $D$, $D\cap L(E)\neq \phi$.
  \item [(iii)]$D$ is said to be \emph{directed} if it is both directed to the right and directed to the left.
\end{enumerate}
\noindent A \emph{directed fuzzy ordered vector space} is a fuzzy vector space which is directed. }
\end{definition}

\begin{theorem}\ \emph{\cite{BI1}} \label{theorem2.3.1}
Let $X$ be a fuzzy ordered vector space, $x, y, z\in X$ and $\alpha, \beta\in R$. Then the following statements hold.
\begin{enumerate}
  \item [(i)]If $\mu(0, x)>1/2$ and $\mu(0, y)>1/2$, then $\mu(0, x+y)>1/2$.
  \item [(ii)]If $\mu(0, x)>1/2$ and $\mu(0, -x)>1/2$, then $x=0$.
  \item [(iii)]If $\mu(0, x)>1/2$ and $\alpha\geq 0$, then $\mu(0, \alpha x)>1/2$.
  \item [(iv)]If $\mu(x_1, x_2)>1/2$ and $\alpha\leq 0$, then $\mu(\alpha x_2, \alpha x_1)>1/2$.
  \item [(v)]If $\mu(0, x)>1/2$ and $\alpha\leq \beta$, then $\mu(\alpha x, \beta x)>1/2$.
\end{enumerate}
\end{theorem}

\begin{theorem}\ \emph{\cite{BI1}} \label{theorem2.3.2}
Let $\{x_j\}_{j\in J}$ be a family of elements in a fuzzy ordered vector space.
\begin{enumerate}
  \item [(i)]If $\lambda\geq 0$, then $\vee_{j\in J} (\lambda x_i)$ exists, and
                \begin{equation*}
                \vee_{j\in J} (\lambda x_i) = \lambda \left(\vee_{j\in J}x_j\right).
                \end{equation*}
  \item [(ii)]If $\lambda<0$, then $\wedge_{j\in J} (\lambda x_i)$ exists, and
                \begin{equation*}
                \wedge_{j\in J} (\lambda x_i) = \lambda \left(\vee_{j\in J}x_j\right).
                \end{equation*}
\end{enumerate}
\end{theorem}

\begin{theorem}\ \emph{\cite{BI1}} \label{theorem2.3.2*}
Let $\{x_j\}_{j\in J}$ and $\{y_l\}_{l\in L}$ be two families of elements in a fuzzy ordered vector space.
If $\vee_{j\in J} x_j $ and $\vee_{l\in L} y_l$ exist, then
\begin{equation*}
\vee_{j\in J, l\in L}(x_j+y_l)=\vee_{j\in J} x_j+\vee_{l\in L} y_l.
\end{equation*}
\end{theorem}

\begin{definition}\ \emph{\cite{BI0}} \label{definition2.3.3}
\emph{A fuzzy ordered vector space is called a \emph{fuzzy Riesz space} if it is also a fuzzy lattice at the same time.}
\end{definition}

\cite{BI0} and \cite{BI1} gave several examples of fuzzy ordered linear spaces and fuzzy Riesz spaces. Below we give an example to show that a fuzzy ordered linear space need not be a fuzzy Riesz space.\\

\noindent \textbf{Example 2.1.} Let $X=D(R)$ be the set of all differential functions on $R$ with coordinate algebraic operations. Define a membership function $\mu: X\times X\rightarrow [0, 1]$ by
\begin{equation*}
\mu(f, g)=\left\{
            \begin{array}{ll}
              1, & \hbox{if $f\equiv g$;} \\
             2/3, & \hbox{if $f(t)\leq g(t)$ for all $t\in R$ and $f\not \equiv g$;} \\
              0, & \hbox{otherwise.}
            \end{array}
          \right.
\end{equation*}
It is routine to verify that $X$ equipped with $\mu$ is a fuzzy ordered linear space. However, $X$ fails to be a fuzzy Riesz space. To see this, take $f(t)=t$ and $g(t)=-t$ in $X$. Put $k(t)=|t|$. Then $\mu(f, k)>1/2$ and $\mu(g, k)>1/2$, that is, $k\in U(\{f, g\})$. If $h\in U(\{f, g\})$, then $\mu(f, h)>1/2$ and $\mu(g, h)>1/2$; hence $f(t)=t\leq h(t)$ and $g(t)=-t\leq h(t)$ for all $t\in [0,1]$, implying $h(t)\geq |t|$ for all $t\in [0, 1]$, that is, $h\in U(k)$. This shows that $f\vee g=k$. But $(f\vee g)(t)=|t|$  is not differentiable at $t=0$. Thus, $f\vee g\not \in X$, proving that $X$ is not a fuzzy Riesz space.\\

\begin{definition}\ \emph{\cite{BI0}} \label{definition2.3.4}
\emph{Let $X$ be a fuzzy Riesz space and $x\in X$. The \emph{positive part} of $x$ is defined by $x^+=x\wedge 0$; the \emph{negative part} of $x^-=(-x)\vee 0$; the \emph{absolute value} of $x$ is defined by $|x|=x\vee (-x)$.}
\end{definition}

\begin{theorem}\ \emph{\cite{BI0}} \label{theorem2.3.3}
Let $X$ be a fuzzy Riesz space and $x, y\in X$. Then the absolute value has the following properties:
\begin{enumerate}
  \item [(i)]$\mu(|x+y|, |x|+|y|)>1/2$;
  \item [(ii)]$|\lambda x|=|\lambda||x|$ for all $\lambda\in R$;
  \item [(iii)]$\mu(|x|-|y|, |x-y|)>1/2$;
  \item [(iv)]$|x-y|=(x\vee y)-(x\wedge y)$.
\end{enumerate}
\end{theorem}

\begin{theorem}\ \emph{\cite{BI0}}  \label{theorem2.3.4}
Let $X$ be a fuzzy Riesz space, $x, y\in X$ and $\lambda>0$. Then the following equalities and inequalities hold.
\begin{enumerate}
  \item [(i)]$\mu((x+y)^+, x^+ + y^+)>1/2$;
  \item [(ii)]$\mu((x+y)^-, x^- + y^-)>1/2$;
  \item [(iii)]$(\lambda x)^+=\lambda x^+$;
  \item [(iv)] $(\lambda x)^-=\lambda x^-$.
\end{enumerate}
\end{theorem}

\begin{theorem}\ \emph{\cite{BI0}}  \label{theorem2.3*}
If $X$ is a fuzzy Riesz space and $x_1, x_2\in X$, then
\begin{equation*}
x_1+x_2=x_1\vee x_2+ x_1\wedge x_2.
\end{equation*}
\end{theorem}

The following theorem is called the \emph{Riesz decomposition theorem} for fuzzy Riesz spaces and the property exhibited in the theorem is called the \emph{Riesz decomposition property} of fuzzy Riesz spaces.

\begin{theorem}\ \emph{\cite{BI0}} \label{theorem2.3.4}
Let $X$ be a fuzzy Riesz space and $x, y_1, ...y_n\in X$. If $\mu(|x|, |y_1+...+y_n|)>1/2$, then there exists elements $x_1,.., x_n\in X$ such that $\mu(|x_i|, |y_i|)>1/2$ for all $i=1, ..., n$ and  $x=x_1+...+x_n$. Moreover, if $x$ is positive, then $x_1, ..., x_n$ can be chosen to be positive.
\end{theorem}

\begin{definition}\ \emph{\cite{BI0}} \label{definition2.3.5}
\emph{Let $X$ be a fuzzy Riesz space.
\begin{enumerate}
  \item [(i)]Two elements $x_1, x_2\in X$ are said to be \emph{disjoint} or \emph{orthogonal}, denoted by
                $x_1\bot x_2$, if $|x_1|\wedge |x_2|=0$.
  \item [(ii)]An element $x\in X$ is said to be \emph{disjoint} or \emph{orthogonal} to a subset $A$ of $X$, denoted by $x\bot A$, if $x\bot y$ for all $y\in A$.
  \item [(iii)]Two subsets $A_1, A_2\in X$ are said to be \emph{disjoint} or \emph{orthogonal}, denoted by
                $A_1\bot A_2$, if $x_1\bot x_2$ for all $x_1\in A_1$ and $x_2\in A_2$.
\end{enumerate}}
\end{definition}

\begin{theorem}\ \emph{\cite{BI0}}  \label{theorem2.3.5}
Let $X$ be a fuzzy Riesz space.
\begin{enumerate}
  \item [(i)]If $x\bot x_1$ and $x\bot x_2$, then $x\bot (ax_1+bx_2)$ for all $a, b\in R$.
  \item [(ii)]If $x=\vee_{j\in J}x_j$ and $y\bot x_j$, then $y\bot x$.
\end{enumerate}
\end{theorem}

\begin{definition}\ \emph{\cite{BI1}}  \label{definition2.3.6}
\emph{A directed ordered fuzzy ordered vector space $X$ is said to be a \emph{fuzzy Arhimedean space} if the set $\{\lambda x\mid \lambda>0\}$ is not bounded above for any nonnegative element $x\in X$. In this case, we also say the space $X$ is \emph{fuzzy Archimedean}.}
\end{definition}
\noindent \textbf{Remark.} A fuzzy Riesz space is directed. Hence, we say a fuzzy Riesz space $X$ is \emph{fuzzy Archimedean} if the set $\{\lambda x\mid \lambda>0\}$ is not bounded above for any nonnegative element $x\in X$.

\begin{theorem}\ \emph{\cite{BI0}}  \label{theorem2.3.7}
Let $X$ be a directed fuzzy ordered vector space. Then $X$ is fuzzy Archimedean if and only if for each nonnegative element $x\in X$ the sequence $\{nx\}_{n\in N}$ is not bounded  above.
\end{theorem}

\begin{theorem}\ \emph{\cite{BI0}}  \label{theorem2.3.6}
Let $X$ be a directed fuzzy ordered vector space. Then $X$ is fuzzy Archimedean if and only if $\wedge_{n\in N}\{1/n\ x\}=0$ for any positive element $x\in X$.
\end{theorem}

We conclude this section by recalling some definitions in linear algebra. Let $V$ be a vector space. An operator $P:V\rightarrow V$ is called a \emph{projection} if $P^2=P$. Let $A_1, A_2$ be two subsets of $V$. Then the \emph{algebraic sum} $A_1+A_2$ is defined as
\begin{equation*}
A_1+A_2=\{x_1+x_2\mid x_1\in A_1, x_1\in A_2\}.
\end{equation*}
If $A_1\cap A_2=\phi$, we write $A_1+A_2$ as $A_1 \oplus A_2$ and call it the \emph{direct sum} of $A_1$ and $A_2$.

\section{Fuzzy Riesz subspaces}

\begin{definition}\label{definition3.1}
\emph{Let $X$ be a fuzzy Riesz space.
\begin{enumerate}
  \item [(i)]A vector subspace $Y$ of $X$ is said to be a \emph{fuzzy Riesz subspace} if for all $x, y\in Y$ the elements $x\vee y$ and $x\wedge y$ belong to $Y$.
  \item [(ii)]A subset $A$ of $X$ is said to be \emph{fuzzy solid} if it follows from $\mu(|x|, |y|)>1/2$ and $y\in A$ that $x\in A$. In this case, we also we $A$ is a \emph{fuzzy solid subset} of $X$.
\end{enumerate}}
\end{definition}
\noindent \textbf{Remark 1.} It is clear from Theorem \ref{theorem2.3.2} that a vector subspace $Y$ of a fuzzy Riesz space $X$ is a fuzzy Riesz subspace if and only if $x, y\in Y$ implies $x\vee y\in Y$.\\

\noindent \textbf{Remark 2.} Every fuzzy solid set $A$ of fuzzy Riesz space is  \emph{circled} (also called \emph{balanced}), that is, $x\in A$ implies $\lambda x\in A$ for all $|\lambda|\leq 1$.\\

The next example shows that a vector subspace of a fuzzy Riesz space need not be a fuzzy Riesz subspace.\\

\noindent \textbf{Example 3.1.} Let $X=C(R)$ be the set of all continuous functions on $R$ with coordinate algebraic operations. Define a membership function $\mu: X\times X\rightarrow [0, 1]$ by
\begin{equation*}
\mu(f, g)=\left\{
            \begin{array}{ll}
              1, & \hbox{if $f\equiv g$;} \\
              2/3, & \hbox{if $f(t)\leq g(t)$ for all $t\in R$ and $f\not \equiv g$;} \\
              0, & \hbox{otherwise.}
            \end{array}
          \right.
\end{equation*}
Then it is easy to see that $X$ is a fuzzy Riesz space. Now let $Y=D(R)$ be the set of all differentiable functions on $R$. Then $Y$ is clearly a vector subspace of $X$. However, Example 2.1 shows that $Y$ is not a fuzzy Riesz subspace of $X$.\\

\begin{definition}\label{definition3.2}
\emph{Let $A$ be a subset of a fuzzy Riesz space $X$. The smallest solid fuzzy subset containing $A$ is called the \emph{fuzzy solid hull} of $A$ and is denote dy $Sol_F(A)$. }
\end{definition}
\noindent \textbf{Remark.} It is easy to see that the fuzzy solid hull $Sol_F(A)$ is given by
\begin{equation*}\label{3.0}
Sol_F(A)=\{x\mid \exists y\in A \text{ such that $\mu(|x|, |y|)>1/2$}\}.
\end{equation*}

\begin{theorem}\label{theorem3.1}
Let $X$ be a fuzzy Riesz space  and $J$ be an arbitrary index set. Then the following two statements hold.
\begin{enumerate}
  \item [(i)]If $Y_1$ is a fuzzy Riesz subspace of $X$ and $Y_2$ is a fuzzy Riesz subspace of $Y_1$, then $Y_2$ is a fuzzy Riesz subspace of $X$.
  \item [(ii)]If $Y_j$ is a fuzzy Riesz subspace of $X$ for each $j\in J$, then $Y=\cap_{j\in J} Y_j$ is a fuzzy Riesz subspace of $X$.
\end{enumerate}
\end{theorem}

\begin{proof}
\begin{enumerate}
  \item [(i)]Let $x, y\in Y_2$. Then $x, y\in X$; hence $z=\sup_X\{x, y\}$ exists, where $\sup_X$ denotes the supremum is taken in $X$. We need to show that $z\in Y_2$. Since $x, y\in Y_1$ and $Y_1$ is a fuzzy Riesz subspace, we have $z\in Y_1$. Therefore, $z\in U_{Y_1}(\{x, y\})$, where the subscript $Y_1$ denotes that the upper bound is taken in $Y_1$. Now let $w\in U_{Y_1}(\{x, y\})$. Then $U_{Y_1}(x, y)(w)>0$, implying $U_{X}(x, y)(w)>0$; hence $w\in U_{X}(x, y)$. This implies that $z\in U_{X}(w)$.
             In view of $z\in Y_1$, we have $z\in U_{Y_1}(w)$. Therefore, $z=\sup_{Y_1}\{x, y\}$. Since $Y_2$ is a fuzzy Riesz subspace of $Y$, we have $z\in Y_2$. This proves that $Y_2$ is a fuzzy Riesz subspace of $X$.
  \item [(ii)]It is evident that $Y$ is a vector subspace of $X$. Let $x, y\in Y$. Then the hypothesis implies $x, y\in Y_j$ for each $j\in J$. Hence, $x\vee y\in Y_j$ for each $j\in J$, showing that $x\vee y\in Y$. Therefore, $Y$ is a fuzzy Riesz subspace of $X$.
   
\end{enumerate}
\end{proof}

\cite{BI1} defined the notion of sequential convergence in fuzzy order relation; \cite{Beg1} further investigated the properties of this mode of convergence. Below we define the notion of convergence of nets in fuzzy order relation and provide some basic properties.

\begin{definition}\label{definition3.3}
\emph{Let $X$ be a foset. A net $\{x_{\alpha}\}_{\alpha\in A}$ in $X$ is said to be \emph{increasing}, denoted by $x_{\alpha}\uparrow$, if $\mu(x_{\alpha}, x_{\beta})>1/2$ when the indices $\alpha$ and $\beta$ satisfy $\alpha\leq \beta$. If in addition $x=\sup_{\alpha\in A}\{x_{\alpha}\}$ exists, then we write $x_{\alpha}\uparrow x$.
%The notation $x_{\alpha}\uparrow \leq x$ means $x_{\alpha}\uparrow$ and $\mu(x_{\alpha},  x)>\frac{1}{2}$ for all $\alpha\in A$.
Likewise, we can define \emph{decreasing} nets in $X$. The notations $x_{\alpha}\downarrow$ and $x_{\alpha}\downarrow x$ should be interpreted similarly.}
\end{definition}

\begin{notation}
Let $X$ be a foset and $D$ be a subset of $X$. We will use the symbol $D\uparrow$ to denote the fact that $D$ is directed to the right; likewise, $D\downarrow$ denotes the fact that $D$ is directed to the left. The symbol $D\uparrow x$ means $D\uparrow$ and $x=\sup D$; similarly, $D\downarrow x$ means $D\downarrow$ and $x=\inf D$.
\end{notation}

\begin{definition}\label{definition3.4}
\emph{A net $\{x_{\alpha}\}_{\alpha\in A}$ in a fuzzy Riesz space $X$ is said to \emph{converge in fuzzy order} to an element $x\in X$, denoted by $x_{\alpha}\xrightarrow{o_F} x$, if there exists another net $\{y_{\alpha}\}_{\alpha\in A}$ such that $\mu(|x_{\alpha}-x|, y_{\alpha})>1/2$ and $y_{\alpha}\downarrow 0$. In this case, $x$ is said to be the \emph{fuzzy order limit} of  $\{x_{\alpha}\}_{\alpha\in A}$. }
\end{definition}

\begin{theorem}\label{theorem3.2}
The fuzzy order convergence has the following properties.
\begin{enumerate}
  \item [(i)]If $x_{\alpha}\xrightarrow{o_F} x$ and $x_{\alpha}\xrightarrow{o_F} y$, then $x=y$. That is, the fuzzy order limit is unique.
  \item [(ii)]A fuzzy order convergent net is bounded.
  \item [(iii)]If $x_{\alpha}\xrightarrow{o_F} x$, $y_{\alpha}\xrightarrow{o_F} y$ and $\mu(x_{\alpha}, y_{\alpha})>1/2$ for all $\alpha\in A$, then $\mu(x, y)>1/2$.
    \item [(iv)]If $x_{\alpha}\uparrow $ (or $x_{\alpha}\downarrow $), then $x_{\alpha}\xrightarrow{o_F} x$
                if and only if If $x_{\alpha}\uparrow x$ ($x_{\alpha}\downarrow x$ respectively).
    \item [(v)]If $x_{\alpha}\xrightarrow{o_F} x$, then any subnet of $x_{\alpha}$ converges to $x$ in fuzzy order.
    \item [(vi)]If $x_{\alpha}\xrightarrow{o_F} x$, $z_{\alpha}\xrightarrow{o_F} x$, and $\mu(x_{\alpha}, y_{\alpha})>1/2$ and $\mu(y_{\alpha}, z_{\alpha})>1/2$ for all $\alpha\in A$, then
                $y_{\alpha}\xrightarrow{o_F} x$.
    \item [(vii)]If $x_{\alpha}\xrightarrow{o_F} x$, then $x_{\alpha}^+\xrightarrow{o_F} x^+$, $x_{\alpha}^-\xrightarrow{o_F} x^-$, and $|x_{\alpha}|\xrightarrow{o_F} |x|$.
    \item [(viii)]$x_{\alpha}\vee y_{\beta} \underset{(\alpha, \beta)}{\xrightarrow{o_F}} x\vee y$ and
                $x_{\alpha}\wedge y_{\beta} \underset{(\alpha, \beta)}{\xrightarrow{o_F}} x\wedge y$.
    \item [(ix)]If $x_{\alpha}\xrightarrow{o_F} x$ and  $y_{\beta}\xrightarrow{o_F} y$, then
                  $ax_{\alpha}+by_{\beta}\underset{(\alpha, \beta)}{\xrightarrow{o_F}} ax+by$, for all $a, b\in R$.
\end{enumerate}
\end{theorem}

\begin{proof}
We show (ix) only since the proofs of (i)-(vii) are completely analogous to the proofs of corresponding results for sequences in \cite{Beg1} and \cite{BI1}.
\begin{enumerate}
  \item [(ix)]Since $x_{\alpha}\xrightarrow{o_F} x$ and  $y_{\beta}\xrightarrow{o_F} y$, there exist
                two nets $\{z_{\alpha}\}$ and $\{w_{\beta}\}$ such that $\mu(|x_{\alpha}-x|, z_{\alpha})>1/2$, $\mu(|y_{\beta}-y|, w_{\beta})>1/2$, $z_{\alpha}\downarrow 0$ and $w_{\beta}\downarrow 0$. By the remark following Definition \ref{definition2.3.1} and Theorem \ref{theorem2.3.3}, we have $\mu((ax_{\alpha}+by_{\beta})-(ax+by) , |a| |x_{\alpha}-x|+|b| |y_{\beta}-y|)>1/2$ and $\mu(|a| |x_{\alpha}-x|+|b| |y_{\beta}-y|, |a| z_{\alpha}+|b|w_{\beta})>1/2$. Therefore,
                
                $\mu((ax_{\alpha}+by_{\beta})-(ax+by), |a| z_{\alpha}+|b|w_{\beta})>1/2$. It is clear that $|a| z_{\alpha}+|b|w_{\beta}\downarrow_{(\alpha, \beta)}0$. Hence, $ax_{\alpha}+by_{\beta}\underset{(\alpha, \beta)}{\xrightarrow{o_F}} ax+by$.
\end{enumerate}
\end{proof}

\begin{definition}
\emph{Let $X$ be a fuzzy Riesz space. The set of all positive elements in $X$ is called the \emph{positive cone} of $X$ and is often denoted by $X^+$, that is, $X^+=\{x\in X\mid \mu(0, x)>1/2\}$.}
\end{definition}

\begin{definition}\label{definition3.5}
\emph{Let $S$ be a subset of a fuzzy Riesz space $X$.
\begin{enumerate}
  \item [(i)]$S$ is said to be \emph{fuzzy $\sigma$-order closed} if it follows from $\{x_n\}_{n\in N}\subset S$ and $x_{n}{\xrightarrow{o_F}} x$ that $x\in S$.
  \item [(ii)]$S$ is said to be \emph{fuzzy order closed} if it follows from $\{x_{\alpha}\}_{\alpha\in A}\subset S$ and $x_{\alpha}{\xrightarrow{o_F}} x$ that $x\in S$.
\end{enumerate}
}
\end{definition}

\begin{theorem}\label{theorem3.3}
Let $S$ be a fuzzy solid subset of fuzzy Riesz space $X$. Then the following two statements hold.
\begin{enumerate}
  \item [(i)]$S$ is fuzzy $\sigma$-order closed if and only if $x_n\uparrow x$ implies $x\in S$ for all increasing sequence $\{x_n\}$ in $S^+$.
  \item [(ii)]$S$ is fuzzy order closed if and only if $x_{\alpha}\uparrow x$ implies $x\in S$ for all increasing net $\{x_{\alpha}\}$ in $S^+$.
\end{enumerate}
\end{theorem}

\begin{proof}
\begin{enumerate}
  \item [(i)]Suppose $\{x_n\}$ is an increasing sequence in $S^+$ such that $x_n\uparrow x$.
             Then Theorem \ref{theorem3.2} (iv) shows that $x_n\xrightarrow{o_F}x$. Since $S$ is fuzzy $\sigma$-order closed,  we have $x\in S$. For the converse, let $\{x_n\}$ be a sequence in $S$ such that $x_n\xrightarrow{o_F}x$. Then there exists a sequence $\{y_n\}$ in $S$ such that
             $\mu(|x_n-x|, y_n)>1/2$ and $y_n\downarrow 0$.
             Thus, $\mu(x-x_n, y_n)>1/2$; this implies $\mu(x, |x_n|+y_n)>1/2$ which further implies
             $\mu(|x|, |x_n|+y_n)>1/2$ 
             It follows that $\mu((|x|-y_n)^+, |x_n|)>1/2$. 
             By the fuzzy solidness of $S$, we have
             $\{(|x|-y_n)^+\}\subset S$. On the other hand, it is clear that  $(|x|-y_n)^+\uparrow |x|$. Hence, the hypothesis implies $|x|\in S$. It follows from the fuzzy solidness of $S$ that $x\in S$,
             proving that $S$ is fuzzy $\sigma$-order closed.
  \item [(ii)]Similar to the proof of (i).
\end{enumerate}
\end{proof}

\section{Fuzzy ideals}

\begin{definition}\label{definition4.1}
\emph{Let $X$ be a fuzzy Riesz space. A fuzzy solid vector subspace $I$ of $X$ is called a \emph{fuzzy ideal} of $X$. }
\end{definition}
\noindent \textbf{Remark 1.} Our definition of fuzzy ideals is different from the notion of weak ideal defined in \cite{Venugopalan1}. It is clear that a fuzzy ideal is a weak ideal while the converse need not be true.\\

\noindent \textbf{Remark 2.} It is easy to see that (iii) is equivalent to saying that a vector subspace $I$ of $X$ is a fuzzy ideal if it satisfies the following two conditions:
\begin{enumerate}
  \item [(1)]$x\in I$ if and only if $|x|\in I$;
  \item [(2)]if $x$ a positive element in $X$, $\mu(x, y)>1/2$ and $y\in I$, then $x\in I$.\\
\end{enumerate}

\noindent \textbf{Remark 3.} It is clear from Definition \ref{definition3.1} that a vector subspace of $I$ of a fuzzy Riesz space $X$ is a fuzzy ideal if it satisfies the following two conditions:
\begin{enumerate}
  \item [(1)]$x\in I$ if and only if $|x|\in I$;
  \item [(2)]if $x, y\in X^+$ and $y\in I$, then $x\wedge y\in I$.
\end{enumerate}

\begin{theorem}\label{theorem4.1}
Let $X$ be a fuzzy Riesz space and $J$ be an arbitrary index set. Then the following two statements hold.
\begin{enumerate}
  \item [(i)]If $I_1$ is a fuzzy ideal of $X$ and $I_2$ is a fuzzy ideal of $I_1$, then $I_2$ is a fuzzy ideal of $X$.
  \item [(ii)]If $I_j$ is a fuzzy ideal of $X$ for each $j\in J$, then $I=\cap_{j\in J} I_j$ is a fuzzy ideal of $X$.
\end{enumerate}
\end{theorem}

\begin{proof}
\begin{enumerate}
  \item [(i)]By Theorem \ref{theorem3.1} (i), we know that $I_2$ is a fuzzy Riesz subspace of $X$. Thus, it suffices to show that $I_2$ is a fuzzy solid subset of $X$. To this end, let $x \in X$ and $y\in I_2$ with $\mu(|x|, |y|)>1/2$. Then $y\in I_1$. Since $I_1$ is a fuzzy ideal of $X$, we have $x\in I_1$. Therefore, we have $x\in I_1, y\in I_2$ and $\mu(|x|, |y|)>1/2$.
      Now the fact that $I_2$ is a fuzzy ideal of $I_1$ implies $x\in I_2$. This shows that that $I_2$ is a fuzzy
        ideal of $X$.
  \item [(ii)]By Theorem \ref{theorem3.1} (ii), $I$ is a fuzzy Riesz subspace of $X$. Let $x\in X$ and $y\in Y$ with $\mu(|x|, |y|)>1/2$. Then $y\in I_j$ for each $j\in J$. Hence, the fuzzy solidness of $I_j$ implies $x\in I_j$ for each $j\in J$, showing that $x\in I$. This proves that $I$ is a fuzzy ideal of $X$.
\end{enumerate}
\end{proof}

\noindent \textbf{Example 4.1.} Let $X=C(R)$ the set of all continuous functions on $R$ with coordinate algebraic operations. Define a membership function $\mu: X\times X\rightarrow [0, 1]$ by
\begin{equation*}
\mu(f, g)=\left\{
            \begin{array}{ll}
              1, & \hbox{if $f(t)\equiv g(t)$;} \\
              2/3, & \hbox{if $f(t)\leq g(t)$ for all $t\in R$ and $f\not\equiv g$;} \\
              0, & \hbox{otherwise.}
            \end{array}
          \right.
\end{equation*}
Then $X$ is a fuzzy Riesz space. Consider $I=L^1(R)$, i.e., the set of all integrable functions on $R$. We claim that $I$ is a fuzzy ideal of $X$. To see this, let $f\in X$ and $g\in Y$ with $\mu(|f|, |g|)>1/2$. Then
the definition of $\mu$ implies either $f\equiv g$  or $f(t)\leq g(t)$, for all $t\in R$ and $f\not \equiv g$. In either case, we have $\int_R|f(t)|dt\leq \int_R|g(t)dt<\infty$, showing that $f$ is integrable, i.e., $f\in I$. Thus, $I$ is a fuzzy ideal of $X$. \\

However, the next two examples show that a fuzzy Riesz subspace need not be a fuzzy ideal.\\

\noindent \textbf{Example 4.2.} Let $X=C(R)$ be the set of all continuous functions on $R$ with coordinate algebraic operations. Define a membership function $\mu: X\times X\rightarrow [0, 1]$ by
\begin{equation*}
\mu(f, g)=\left\{
            \begin{array}{ll}
              1, & \hbox{if $f\equiv g$;} \\
             2/3, & \hbox{if $f(t)\leq g(t)$ for all $t\in R$ and $f\not \equiv g$;} \\
              0, & \hbox{otherwise.}
            \end{array}
          \right.
\end{equation*}
Then $X$ is a fuzzy Riesz space. Consider $Y=\{f\mid \text{$f$ is a constant function on $R$}\}$. Then $Y$ is clearly a fuzzy Riesz subspace of $X$. But $Y$ is not a fuzzy ideal of $X$. To see this, let
\begin{equation*}
f(t)=\left\{
       \begin{array}{ll}
         1-e^{-t}, & \hbox{if $t\geq 0$;} \\
         0, & \hbox{otherwise.}
       \end{array}
     \right.
\end{equation*}
and $g(t)=2$ for all $t\in R$. Then $\mu(|f|, |g|)>1/2$ and $g\in Y$. However, $f\not\in Y$.\\

\noindent \textbf{Example 4.3.} Let $X=R^R$ be the set of all real-valued functions on $R$ with coordinate algebraic operations. Define a membership function $\mu: X\times X\rightarrow [0, 1]$ by
\begin{equation*}
\mu(f, g)=\left\{
            \begin{array}{ll}
              1, & \hbox{if $f\equiv g$;} \\
              2/3, & \hbox{if $f(t)\leq g(t)$ for all $t\in R$ and $f\not \equiv g$;} \\
              0, & \hbox{otherwise.}
            \end{array}
          \right.
\end{equation*}
Then $X$ is a fuzzy Riesz space. Consider $Y=\{f\mid \text{$f$ is a continuous function on $R$}\}$. Then $Y$ is clearly a fuzzy Riesz subspace of $X$. But $Y$ is not a fuzzy ideal of $X$. To see this, put
\begin{equation*}
f(t)=\left\{
       \begin{array}{ll}
         1, & \hbox{if $t\geq 0$;} \\
         -1, & \hbox{if $t<0$.}
       \end{array}
     \right.
\end{equation*}
and $g(t)=2$ for all $t\in R$. Then $\mu(|f|, |g|)>1/2$ and $g\in Y$. However, $f\not\in Y$.\\

\begin{definition}\label{definition4.4}
\emph{Let $D$ be a subset of a fuzzy Riesz space $X$. The smallest fuzzy ideal of $X$ that contains $D$ is called the \emph{fuzzy ideal generated by $D$} and is denoted by $I_D$. If $D$ is a singleton, that is, $D=\{x\}$ for some $x\in X$, then $I_D$ is often written as $I_x$ and is called the \emph{principal fuzzy ideal} generated by $x$. }
\end{definition}

\begin{theorem}\label{theorem4.2}
Let $D$ be a subset of a fuzzy Riesz space $X$.
\begin{enumerate}
  \item [(i)]$I_D$ exists and is unique.
  \item [(ii)]$I_D$ can be descried as follows.
             \begin{equation}\label{3.1}
             I_D=\{x\in X \mid \exists x_1, ..., x_n\in D \text{ and $\lambda\geq 0$ such that $\mu(|x|, \lambda\sum_{i=1}^n|x_i|)>1/2$}\}.
             \end{equation}
\end{enumerate}
\end{theorem}

\begin{proof}
\begin{enumerate}
  \item [(i)]By Theorem \ref{theorem4.1}, the intersection of all fuzzy ideals containing $D$ is a fuzzy ideal. Clearly,
             this fuzzy ideal is unique and it is the smallest fuzzy ideal that contains $D$.
  \item [(ii)]Let $\widetilde{I}$ denotes the set on the right-hand side of Equation (\ref{3.1}). By taking $n=1$ and $x_1=x$, we know $D\subset \widetilde{I}$. If $x\in X$ and $y\in \widetilde{I}$ with $\mu(|x|, |y|)>1/2$, then there exist $x_1, ..., x_n\in D$ and $\lambda\geq 0$ such that
            $\mu(|y|, \lambda\sum_{i=1}^n |x_i|)>1/2$. It follows that
            $\mu(|x|,  \lambda\sum_{i=1}^n |x_i|)>1/2$, implying $x\in \widetilde{I}$. This shows that $\widetilde{I}$ is a fuzzy ideal containing $D$. Hence, $I_D\subset \widetilde{I}$. Conversely, for $x\in \widetilde{I}$, there  exist $x_1, ..., x_n\in D$ and $\lambda\geq 0$ such that
            $\mu(|x|, \lambda\sum_{i=1}^n |x_i|)>1/2$. Thus, $x\in I_D$.
            This shows that $\widetilde{I}\subset I_D$.
            Therefore, $I_D=\widetilde{I}$, establishing (\ref{3.1}).
\end{enumerate}
\end{proof}

\begin{corollary}\label{corollary4.1}
Let $X$ be a fuzzy Riesz space and $y\in X$. Then principal fuzzy ideal $I_x$ can be described as
             \begin{equation}
             I_x=\{y\in X \mid \exists \lambda\geq 0  \text{ such that $\mu(|y|, \lambda|x|)>1/2$}\}\nonumber.
             \end{equation}
\end{corollary}

\begin{theorem}\label{theorem4.3}
Let $X$ be a fuzzy Riesz space and $I_1, I_2$ be two fuzzy ideals of $X$. Then the following statements hold.
\begin{enumerate}
  \item [(i)]$I_1+I_2$ is a fuzzy ideal of $X$.
  \item [(ii)]$I_1^+ + I_2^+ = (I_1+I_2)^+$.
  \item [(iii)]If $I_1\cap I_2=\phi, x=x_1+x_2, y=y_1+y_2$, where $x_1, y_1\in I_1$ and $x_2, y_2\in I_2$, then $\mu(x, y)>1/2$ implies $\mu(x_1, y_1)>1/2$ and $\mu(x_2, y_2)>1/2$.
\end{enumerate}
\end{theorem}

\begin{proof}
\begin{enumerate}
  \item [(i)]Let $x\in X$ and $y\in I_1+I_2$ with $\mu(|x|, |y|)>1/2$. Write $y=y_1+y_2$, where $y_1\in I_1$
             and $y_2\in I_2$.
              Since $\mu(x^+, |x|)>1/2$ and $\mu(|y|, |y_1+y_2|)>1/2$, we have
              $\mu(x^+, |y_1|+|y_2|)>1/2$.
              A fuzzy Riesz space has the Riesz decomposition property; therefore,
              there exist two positive elements $x_1, x_2$ such that $\mu(x_1, |y_1|)>1/2,
              \mu(x_2, |y_2|)>1/2$ and $x=x_1+x_2$. As $y_1\in I_1$ and $y_2\in I$,
              we have $x_1\in I_1$ and $x_2\in I_2$, showing that $x\in I_1+I_2$. Thus, $I_1+I_2$ is a fuzzy ideal of $X$.
  \item [(ii)]Take $x\in (I_1+I_2)^+$. Then $x=x_1+x_2$, where $x_1\in I_1$ and $x_2\in I_2$.
                  We have $\mu(x, |x_1|+|x_2|)>1/2$. Thus, the Riesz decomposition theorem implies
                  that there exist positive elements $\widetilde{x}_1$ and $\widetilde{x}_2$ in $X$ such that
                  $\mu(\widetilde{x}_1, |x_1|)>1/2, \mu(\widetilde{x}_2, |x_2|)>1/2$ and $x=\widetilde{x}_1+\widetilde{x}_2$. Since $x_1\in I_1$ and $x_2\in I_2$, it follows that
                  $\widetilde{x}_1\in I_1^+$ and $\widetilde{x}_2\in I_2^+$. Thus, $x=\widetilde{x}_1+\widetilde{x}_2\in I^+_1 + I_2^+$. This shows that $(I_1+I_2)^+\subset I_1^+ + I_2^+$. It is obvious that $I_1^+ + I_2^+\subset (I_1+I_2)^+$. Therefore, $I_1^+ + I_2^+= (I_1+I_2)^+$.
  \item [(iii)]Since $I_1\cap I_2=\phi$, we have a unique decomposition $y-x=(y_1-x_1)+(y_2-x_2)$,
                where $y_1-x_1\in I_1$ and $y_2-x_2\in I_2$. By the hypothesis $\mu(x, y)>1/2$,
                 we know $y-x\in (I_1+I_2)^+$. It follows from (ii) that $\mu(x_1, y_1)>1/2$ and $\mu(x_2, y_2)>1/2$.
\end{enumerate}
\end{proof}

\begin{definition}\label{definition4.4}
\emph{Let $X$ be a fuzzy Riesz space and $A\subset X$. The set
\begin{equation*}
A^d=\{x\in X\mid x\bot y,\quad \forall y\in A\}
\end{equation*}
is called the \emph{disjoint complement} of $A$. The notation $A^{dd}$ denotes
the disjoint complement of $A^d$, i.e., $A^{dd}=(A^d)^d$. Notations $A^{ddd}, A^{dddd}, ...$ should
be interpreted in the same manner.}
\end{definition}
\noindent \textbf{Remark.} Evidently, if $A_1$ and $A_2$ are two subsets of a fuzzy Riesz space such that $A_1\subset A_2$, then
$A_2^d\subset A_1^d$. \\

\begin{theorem}\label{theorem4.4}
Let $A$ be a subset of a fuzzy Riesz space $X$. Then the following statements hold.
\begin{enumerate}
  \item [(i)]$A\subset A^{dd}$.
  \item [(ii)]$A^d=A^{ddd}$.
  \item [(iii)]$A^d\cap A^{dd}=\{0\}$.
  \item [(iv)]If $A^d=\{0\}$, then $A^{dd}=X$.
  \item [(v)]$A^d$ is a fuzzy ideal of $X$.
  \item [(vi)]If $A$ is a fuzzy ideal of $X$, then for every nonzero element $x\in A^{dd}$ there
               exists a nonzero element $y\in A$ such that $\mu(|y|, |x|)>1/2$.
\end{enumerate}
\end{theorem}

\begin{proof}
\begin{enumerate}
  \item [(i)]Let $x\in A$. Then for all $y\in A^d$, we have $x\bot y$. Thus, $x\in A^{dd}$.
  \item [(ii)]It is obvious from (i) that $A^d\subset A^{ddd}$. Conversely, it follows from (i) and the remark following Definition \ref{definition4.4} that $A^{ddd}\subset A^d$. Therefore, $A=A^{ddd}$.
    \item [(iii)]It is clear from Definition \ref{definition2.3.5} that $0\bot x$ for all $x\in X$.
                 Thus, $\{0\}\subset A^d \cap A^{dd}$. For the converse, let $x\in A^d\cap A^{dd}$.
                 Then the definition of $A^{dd}$ implies $|x|\bot |x|$. Therefore, $x=0$, showing that
                 $A^d\cap A^{dd}\subset \{0\}$. It follows that $A^d\cap A^{dd}=\{0\}$.
    \item [(iv)]Since $0\bot x$ for all $x\in X$, we have $\{0\}^d=X$ by Definition \ref{definition2.3.5}. Therefore, the conclusion follows.
    \item [(v)]It is clear from Theorem \ref{theorem2.3.2} and Theorem \ref{theorem2.3.5} that $A^d$
            is a  vector subspace of $X$. Let $x\in X$ and $y\in A^d$ with $\mu(|x|, |y|)>1/2$.
            Then $|y|\wedge |z|=0$ for all $z\in A$. Since $|x|\wedge |y|=|x|$ by Theorem \ref{theorem2.1.3}, for all $z\in A$ we have
            \begin{equation*}
            |x|\wedge |z|=(|x|\wedge |y|)\wedge |z|=|x|\wedge (|y|\wedge|z|)=0.
            \end{equation*}
            Thus, $x\in A^d$. This shows that $A^d$ is a fuzzy ideal.
    \item [(vi)]Suppose not. Let $x\in A^{dd}$ and $x\neq 0$. If there exists some $z\in A$ such that
                $|x|\wedge |z|\neq 0$, then the fuzzy solidness of $A$ implies $|x|\wedge |z|\in A$.
                It is evident that $\mu(|x|\wedge |z|, |x|)>1/2$; this contradicts our hypothesis. Thus, $|x|\wedge |z|=0$ for all $z\in A$, that is $x\in A^d$. It follows from (iii) that
                $x=0$, contradicting the hypothesis that $x\neq 0$. Therefore, $A^{dd}$ must possess the stated property.
\end{enumerate}
\end{proof}
\noindent \textbf{Remark.} It is clear that the proof of (iii) yields a slightly stronger statement:
\emph{If $A$ and $B$ be two disjoint subsets of a fuzzy Riesz space, then either $A\cap B=\phi$ or $A\cap B=\{0\}$.
}\\
\begin{theorem}\label{theorem4.5}
Let $I$ be a fuzzy ideal of a fuzzy Riesz space $X$.
\begin{enumerate}
  \item [(i)]$I^{dd}$ is the largest fuzzy ideal $\widetilde{I}$ in $X$ having the property that
             for every nonzero element $x\in \widetilde{I}$ there exists a nonzero element $y\in I$ such that
             $\mu(|y|, |x|)>1/2$.
  \item [(ii)]$I^d=\{0\}$ if and only if for every nonzero element $x\in X$ there exists a
                nonzero element $y\in I$ such that $\mu(|y|, |x|)>1/2$.
\end{enumerate}
\end{theorem}

\begin{proof}
\begin{enumerate}
  \item [(i)]By Theorem \ref{theorem4.4} (vi), $I^{dd}$ has the stated property.
             It suffices to show that $I^{dd}$ is the largest fuzzy ideal having the stated property.
             Suppose not. Let $\widetilde{I}$ be a fuzzy ideal of $X$ with the
                stated property. Then there exist $x\in \widetilde{I}$ and $y\in I^d$ such that
                $|x| \wedge |y|\neq 0$. Since $(|x| \wedge |y|, |x|)>1/2$, $(|x| \wedge |y|, |y|)>1/2$ and $I^d$ and $\widetilde{I}$ are both fuzzy solid, we have $|x| \wedge |y|\in  \widetilde{I}\cap I^d$. From the hypothesis we may find some nonzero element $z\in I$ such that such that $\mu(|z|, |x|\wedge |y|)>1/2$.
                Since $\widetilde{I}\cap I^d$ is a fuzzy ideal of $X$,
                we have $z\in I\cap I^d=\{0\}$, contradicting $z\neq 0$. Therefore, $I^{dd}$ is the largest fuzzy ideal having the stated property.
  \item [(ii)]The conclusion follows readily from (i) and Theorem \ref{theorem4.4} (iv).
\end{enumerate}
\end{proof}

\begin{definition}\label{definition4.5}
\emph{Let $X$ be a fuzzy Riesz space. A fuzzy Riesz subspace $Y$ of $X$ is said to be \emph{fuzzy order dense} in $X$ if for every nonzero positive element $x\in X$ there exists a nonzero element $y\in Y$ such that $\mu(y, x)>1/2$.}
\end{definition}

\begin{theorem}\label{theorem4.6}
Let $X$ be a fuzzy Riesz space and $I$ be a fuzzy ideal of $X$. Then the following statements hold.
\begin{enumerate}
  \item [(i)]$I$ is fuzzy order dense in $X$ if and only if $I^d=\{0\}$.
  \item [(ii)]$I\oplus I^d$ is fuzzy order dense ideal of $X$.
  \item [(iii)]$I$ is fuzzy order dense in $I^{dd}$.
\end{enumerate}
\end{theorem}

\begin{proof}
\begin{enumerate}
  \item [(i)]Suppose $I$ is fuzzy order dense in $X$ and let $x\in X^+\cap I^d$.
             If $x\neq 0$, then there exists some nonzero element $y\in I$ such that $\mu(y, x)>1/2$.
             Hence, $y\in I\cap I^d=\{0\}$, implying $y=0$; this contradicts the hypothesis that $y\neq 0$. Thus, we must have $x=0$, i.e., $I^d=\{0\}$. Conversely, suppose $I^d=\{0\}$. Let $x$ be a nonzero element in $X^+$. If $x\wedge y=0$ for all $y\in I^+$, then
             $x$ would belong to $I^d=\{0\}$, contradicting the hypothesis that $x\neq 0$. Thus, there exists
             some $y\in I^+$ such that $x\wedge y\neq 0$. Since $\mu(x\wedge y, x)>1/2$ and $I$ is  fuzzy solid, we have $x\wedge y\in I$. This proves that $I$ is fuzzy order dense.
  \item [(ii)]By Theorem \ref{theorem4.3} (i) and the remark following Theorem \ref{theorem4.4}, we know
                that $I \oplus I^d$ is a fuzzy ideal of $X$. Next, take $x\in (I \oplus I^d)^d$. Then
                $x\bot I$ and $x\bot I^d$, implying $x\in I^d$ and $I\cap I^{dd}$, respectively. Thus, $x=0$, showing that $(I \oplus I^d)^d=\{0\}$. It follows from (i) that $I \oplus I^d$ is fuzzy order dense in $X$.
  \item [(iii)]Since the disjoint complement of $I$ in $I^{dd}$ is $I^d\cap I^{dd}$, the conclusion follows from (i) and Theorem \ref{theorem4.4} (iii).
\end{enumerate}
\end{proof}

\section{Fuzzy bands}

\begin{definition}\label{definition5.1}
\emph{Let $X$ be a fuzzy Riesz space and $I$ be a fuzzy ideal of $X$.
\begin{enumerate}
  \item [(i)]If $I$ is fuzzy $\sigma$-order closed, we say $I$ is a \emph{fuzzy $\sigma$-ideal} of $X$.
            .
  \item [(ii)]If $I$ is fuzzy order-closed, we say $I$ is a \emph{fuzzy band} of $X$.
  \end{enumerate}}
\end{definition}
\noindent \textbf{Remark.} It is clear that a fuzzy ideal $B$ of a fuzzy Riesz space $X$ is a fuzzy band if and only if $D\subset B$ and $x=\sup D$ implies $x\in B$.\\

Theorem \ref{theorem3.3} immediately implies the following theorem.

\begin{theorem}\label{theorem5.0}
Let $I$ be a fuzzy ideal of a fuzzy Riesz space $X$. Then the following two statements hold.
\begin{enumerate}
  \item [(i)]$I$ is a fuzzy $\sigma$-ideal if and only if $x_n\uparrow x$ implies $x\in I$ for all increasing sequence $\{x_n\}$ in $I^+$.
  \item [(ii)]$I$ is a band if and only if $x_{\alpha}\uparrow x$ implies $x\in I$ for all  increasing net $\{x_{\alpha}\}$ in $I^+$.
\end{enumerate}
\end{theorem}

\noindent \textbf{Example 5.1.} Consider $X=L^1([0, 1])$ the set of all integrable functions on $[0, 1]$ with coordinate algebraic operations. Define a membership function $\mu: X\times X\rightarrow [0, 1]$ by
\begin{equation*}
\mu(f, g)=\left\{
            \begin{array}{ll}
              1, & \hbox{if $f\equiv g$;} \\
              2/3, & \hbox{if $f(t)\leq g(t)$ for all $t\in [0, 1]$ and $f\not\equiv  g$;} \\
              0, & \hbox{otherwise.}
            \end{array}
          \right.
\end{equation*}
Then $X$ is a fuzzy Riesz space. Let $B=\{f\in L^1([0, 1])\mid f(x)=0\ a.e. \text{ on $[0,1]$}\}$, that is, the set of almost zero integrable functions on $[0,1]$. We claim that $B$ is a fuzzy band of $X$. To see this, let $h\in X$ and $g\in I$ such that $\mu(|h|, |g|)>1/2$. Then we have $0\leq |h(t)|\leq |g(t)|$ for all $t\in [0, 1]$. Since $g=0\ a.e.$ on $[0, 1]$, it follows $h=0\ a.e.$ on $[0, 1]$. Thus, $h\in B$, showing that $B$ is a fuzzy ideal of $X$. Next, let $\{f_{\alpha}\}$ be a net in $B$ such that $f_{\alpha}\xrightarrow{o_F} f$.
Then there exists a net $\{g_{\alpha}\}$ in  $X$ such that $\mu(|f_{\alpha}-f|, |g_{\alpha}|)>1/2$ and $g_{\alpha}\downarrow 0$.

If $f\not \in B$, that is, $f\neq 0\ a.e$ on $[0, 1]$, then there exists a positive integer $m$ such that
such that $\nu(E_m)>0$, where $E_m=\{x\in [0, 1]\mid f(t)>1/m\}$ and $\nu$ is the Lebesgue measure on $[0, 1]$. Take a sequence $\{f_n\}$ of $\{f_{\alpha}\}$ such that $g_n\downarrow g$. Then we have $f_{n}\xrightarrow{o_F} f$ and $\mu(|f_n-f|, |g_n|)>1/2$. Also, there exists a Lebesgue measurable subset $F_m\subset E_m$ such that $\nu(F_m)>0$, $f_n\equiv 0$ on $F_m$ for all $n\in N$, and $f(t)>1/m$ for all $t\in F$. But each $f_n=0 \ a.e.$ on $[0, 1]$; hence we have $|f_n(t)-f(t)|> 0\ a.e. $ on $F_m$ for all $n\in N$. Therefore, $\mu(1/(m+1), |f_n-f|)=2/3$, implying $\mu(1/(m+1), |g_n|)>1/2$.
It follows that $1/(m+1)\in L(\{g_n\})$, contradicting the fact that $\inf\{g_n\}=0$.\\

The next example shows that a fuzzy ideal need not be a fuzzy band.\\

\noindent \textbf{Example 5.2.}
 Consider $X=R^R$ the set of all real-valued functions on $R$ with coordinate algebraic operations. Define a membership function $\mu: X\times X\rightarrow [0, 1]$ by
\begin{equation*}
\mu(f, g)=\left\{
            \begin{array}{ll}
              1, & \hbox{if $f\equiv g$;} \\
              2/3, & \hbox{if $f(t)\leq g(t)$ for all $t\in R$ and $f\neq g$;} \\
              0, & \hbox{otherwise.}
            \end{array}
          \right.
\end{equation*}
Then $X$ is a fuzzy Riesz space. Let $I=\{f\in X \mid f(0)=0\}$, i.e., the set of real-valued functions vanishing at $0$. We claim that $I$ is a fuzzy ideal but not a fuzzy band of $X$. To see this, let $f\in X$ and $g\in I$ such that $\mu(|f|, |g|)>1/2$. If $\mu(|f|, |g|)=1$, then $|f|=|g|$; hence $f(0)=g(0)$ showing that $f\in I$. If $\mu(|f|, |g|)=2/3$, then $|f(0)|\leq |g(0)|=0$ implying $f(0)=0$; it follows that $f\in I$.
In either case, we have $f\in I$; therefore, $I$ is a fuzzy ideal of $X$. However, $I$ fails to be a fuzzy band of $X$.
To see this, consider $D=\{f_n\}_{n\in N}$, where $f_n$ is defined as
\begin{equation*}
f_n(t)=\left\{
         \begin{array}{ll}
           nt, & \hbox{if $t\leq 1/n$;} \\
           1, & \hbox{otherwise.}
         \end{array}
       \right.
\end{equation*}
Let $f(t)=1$ for all $t\in [0, 1]$. Then $f_n(t)\leq f(t),$ for all $t\in [0, 1]$. Thus, $\mu(f_n, f)>1/2$ for all $n\in N$, implying $f\in U(D)$. Now for any $g\in U(D)$ we have $\mu(f_n, g)>1/2$ for all $n\in N$. By definition of $\mu$, we know $f_n(t)\leq g(t)$ for all $t\in [0, 1]$ and all $n\in N$ ; hence $1\leq g(t)$ for all $t\in [1/n, 1]$ and all $n\in N$, implying $f\equiv 1\leq g(t)$ on $[0, t]$. Thus, $\mu(f, g)>1/2$, showing that $g\in U(f)$. Therefore, $f=\sup D$. Since $f(0)=1\neq 0$, $f\not\in I$. This shows that $I$ is not a fuzzy band of $X$.

\begin{theorem}\label{theorem5.1}
Let $X$ be a fuzzy Riesz space and $J$ be an arbitrary index set. Then the following two statements hold.
\begin{enumerate}
  \item [(i)]If $B_1$ is a fuzzy band of $X$ and $B_2$ is a fuzzy band $B_1$, then $B_2$ is a fuzzy band of $X$.
  \item [(ii)]If $B_j$ is a fuzzy band of $X$ for each $j\in J$, then $B=\cap_{j\in J} B_j$ is a fuzzy band of $X$.
\end{enumerate}
\end{theorem}

\begin{proof}
\begin{enumerate}
  \item [(i)]By Theorem \ref{theorem4.1} (i), $B_2$ is a fuzzy ideal of $X$. It remains to show that $B_2$ is fuzzy order-closed in $X$. To this end, let $\{x_{\alpha}\}$ be a net in $B_2$ such that
            $x_{\alpha}\xrightarrow{o_F} x$ in $X$. Since $B_2\subset B_1$ and $B_1$ is a fuzzy band of $X$, we have
            $x\in B_1$. Thus, $x_{\alpha}\xrightarrow{o_F} x$ in $B_1$. As $B_2$ is a fuzzy band of $B_1$,
            we have $x\in B_2$, proving that $B_2$ is a fuzzy band of $X$.
  \item [(ii)]By Theorem \ref{theorem4.1} (ii), $B=\cap_{j\in J} B_j$ is a fuzzy ideal of $X$. Since the intersection of fuzzy order-closed sets is obviously fuzzy order-closed, $B$ is a fuzzy band of $X$.
\end{enumerate}
\end{proof}

\begin{definition}\label{definition5.2}
\emph{Let $D$ be a subset of a fuzzy Riesz space $X$. The smallest band in $X$ that contains $D$ is called the \emph{fuzzy band generated by $D$} and is denoted by $B_D$. If $D$ is a singleton, that is, $D=\{x\}$ for some $x\in D$, then $B_D$ is often written as $B_x$ and is called the \emph{principal fuzzy band} generated by $x$. }
\end{definition}

\begin{theorem}\label{theorem5.2}
Let $D$ be a subset of a fuzzy Riesz space $X$.
\begin{enumerate}
  \item [(i)]$B_D$ exists and is unique.
  \item [(ii)]$B_D$ can be descried as follows.
             \begin{equation}\label{5.1}
             B_D=\{x\in X \mid \text{ there exists a net $\{x_{\alpha}\}_{\alpha\in A}\in I_D^+$ such that $x_{\alpha}\uparrow |x|$}\}.\nonumber
             \end{equation}
\end{enumerate}
\end{theorem}

\begin{proof}
\begin{enumerate}
  \item [(i)]Theorem \ref{theorem5.1} shows that the intersection of all fuzzy bands containing $D$ is a fuzzy ideal. Clearly,
             this fuzzy band is unique and it is the smallest fuzzy band that contains $D$.
  \item [(ii)]Let $\widetilde{B}$ denote the left-hand side of (\ref{5.1}). From Theorem \ref{theorem5.0}, we see that if a fuzzy band $B$ contains $D$, then it evidently contains $\widetilde{B}$. Also, it is clear
            that $D\subset \widetilde{B}$. Thus, it suffices to show that $\widetilde{B}$ is a fuzzy band of $X$. To this end, let $x, y\in  \widetilde{B}$. Then there are two nets $\{x_{\alpha}\}_{\alpha\in A}$ and $\{y_{\beta}\}_{\beta\in B}$ such that $x_{\alpha}\uparrow |x|$ and $y_{\beta}\uparrow |y|$.
            For indices $\alpha_1\leq \alpha_2$ and $\beta_1\leq \beta_2$, we have
            $\mu(x_{\alpha_1}, x_{\alpha_2})>1/2$ and $\mu(y_{\beta_1}, y_{\beta_2})>1/2$. It follows
            that $\mu(x_{\alpha_1}+y_{\beta_1}, x_{\alpha_2}+y_{\beta_2})>1/2$, that is, $x_{\alpha}+y_{\beta}\uparrow_{\alpha, \beta}$.
            Then  Theorem \ref{theorem3.2} shows that $|x+y|\wedge(x_{\alpha}+y_{\beta})\uparrow_{\alpha, \beta} |x+y|$. Similarly, we have $|\lambda|x_{\alpha}\uparrow |\lambda x|$ for each $\lambda\in R$. Therefore, $\widetilde{B}$ is a vector subspace of $X$.   Next, let $z\in X$ such that $\mu(|z|, |x|)>1/2$. Since $\mu(|z|\wedge x_{\alpha}, x_{\alpha})>1/2$ and $\mu(x_{\alpha}, x)>1/2$, we have $\mu(|z|\wedge x_{\alpha}, x)>1/2$ for each $\alpha$. It follows from the fuzzy solidness of $I_D$ that  $\{|z|\wedge x_{\alpha}\}\subset I_D^+$. Clearly, the net $\{|z|\wedge x_{\alpha}\}$ is increasing. Hence,  Theorem \ref{theorem3.2} implies that $|z|\wedge x_{\alpha}\uparrow |z|$. Thus, $z\in B$, showing that $B$ is a fuzzy ideal of $X$. Finally, let $\{w_{\alpha}\}_{\alpha\in A}\subset \widetilde{B}^+$ such that $w_{\alpha}\uparrow w$. Define $E=\{v\in I_D^+\mid \mu(v, w_{\alpha})>1/2 \text{ for some $\alpha\in A$}\}$. Then $E\subset I_D^+$ and
            $\sup(E)=\sup_{\alpha\in A}\{\sup E_{\alpha}\}$, where $E_{\alpha}=\{v\in I_D^+\mid \mu(v, w_{\alpha})>1/2\}$.
            Therefore, $\widetilde{B}$ is a fuzzy band of $X$, establishing $B_D=\widetilde{B}$.
\end{enumerate}
\end{proof}

\begin{corollary}\label{corollary5.1}
Let $X$ be a fuzzy Riesz space and $x\in X$. Then principal fuzzy band $B_x$ can be described as
             \begin{equation}
             B_x=\{y\in X \mid |y|\wedge (n|x|)\uparrow |y|\}.\nonumber
             \end{equation}
\end{corollary}

\begin{proof}
Let $y\in D_x$ and $I_x$ be the principal fuzzy ideal generated by $x$. Then Theorem \ref{theorem5.2} shows that there exists a net $\{y_{\alpha}\}_{\alpha\in A} \subset I_x$ such that $y_{\alpha}\uparrow y$. It follows from Theorem \ref{theorem4.2} that for each $\alpha\in A$ there exists a positive integer $n$ such that $\mu(y_{\alpha}, n|x|)>1/2$. Since $y=\sup\{y_{\alpha}\}$, we have $\mu(y_{\alpha}, y)>1/2$ for all $\alpha\in A$. Thus, $\mu(y_{\alpha}, y\wedge n|x|)>1/2$ for all $\alpha\in A$ and $\mu(y\wedge n|x|, |y|)>1/2$. In view of the fact that $y_{\alpha}\uparrow y$, we conclude that $y\wedge n|x|\uparrow |y|$. This completes the proof.
\end{proof}

The next theorem shows that a disjoint complement in a fuzzy Riesz space is always a fuzzy band.

\begin{theorem}\label{theorem5.3}
If $A$ is a subset of a fuzzy Riesz space $X$. Then $A^d$ is a fuzzy band in $X$.
\end{theorem}

\begin{proof}
The theorem follows from Theorem \ref{theorem4.4} (v) and Theorem \ref{theorem2.3.5} (ii).
\end{proof}

Theorem \ref{theorem4.3} says that the sum of two fuzzy ideals is a fuzzy ideal. However, the sum of two fuzzy
bands need not be a fuzzy band as the next example shows.\\

\noindent \textbf{Example 5.3.} Let $a$ be a fixed positive number. Consider $X=C([-a, a])$ the set of all continuous functions on $[-a, a]$ with coordinate algebraic operations. Define a membership function $\mu: X\times X\rightarrow [0, 1]$ by
\begin{equation*}
\mu(f, g)=\left\{
            \begin{array}{ll}
              1, & \hbox{if $f\equiv g$;} \\
              2/3, & \hbox{if $f(t)\leq g(t)$ for all $t\in [-a, a]$ and $f\neq g$;} \\
              0, & \hbox{otherwise.}
            \end{array}
          \right.
\end{equation*}
Then $X$ is a fuzzy Riesz space. Let $B_1=\{f\in X\mid f(t)=0 \text{ for all $t\in [0, a]$}\}$ and
$B_1=\{f\in X\mid f(t)=0, \text{ for all $t\in [-a, 0]$}\}$. We claim that $B_1$ and $B_2$ are fuzzy bands in $X$. To see this, let $\{f_{\alpha}\}\subset  B_1$ such that $\{f_{\alpha}\}\uparrow f$. In view of Theorem \ref{theorem5.0}, we need to show that $f\in B_1$. Suppose not. Then there exists $b\in [0, a]$ such that $f(b)\neq 0$. Without loss of generality, we may assume $b\neq 0$ and $f(b)>0$. By the continuity of $f$, there exists $\epsilon>0$ such that $f(t)\neq 0$ for all $t\in [b-\epsilon, b+\epsilon]\subset [0, a]$. Let $m=\max_{t\in [b-\epsilon, b+\epsilon]}f(t)$ and take a number $c$ such that $0<c<\min\{m, f(b)\}$. Then define a function $g_1\equiv c$ on $[b-\epsilon, b+\epsilon]$ and extend it continuously to a nonnegative function on $[0, a]$ using the Tietze Extension Theorem. Next, define a function $g$ on $[-a, a]$ by
\begin{equation*}
g(t)=\left\{
       \begin{array}{ll}
         g_1(t), & \hbox{if $t\in [0, a]$;} \\
         f(t), & \hbox{otherwise.}
       \end{array}
     \right.
\end{equation*}
Then $g\in B_1$ and $\mu(g, f)>1/2$, showing that $g\not\in U(f)$. It is obvious that
$\mu(f_{\alpha}, g)>1/2$ for all $\alpha$, that is, $g\in U(\{f_{\alpha}\})$. But $g$ is strictly less than $f$ on the interval $[b-\epsilon, b+\epsilon]$; this means $g\not \in U(f)$, contradicting $f=\sup\{f_{\alpha}\}$. It follows by contraposition that $f\in B_1$. Therefore, $B_1$ is a fuzzy band of $X$.
Similarly, we can show that $B_2$ is a fuzzy band of $X$.

Evidently, $B_1\cap B_2=\{0\}$; hence Theorem \ref{theorem4.3} shows that $B_1+B_2=B_1\oplus B_2=\{f\in X\mid f(0)=0\}$ is a fuzzy ideal of $X$. However, $B_1+B_2$ is not a fuzzy band. To see this, consider a sequence of function $\{f_n\}$ in $X$ defined by
\begin{equation*}
f_n(t)=\left\{
         \begin{array}{ll}
           1, & \hbox{if $1/n\leq t\leq a$;} \\
           nt, & \hbox{if $0\leq t<1/n$;} \\
           -nt, & \hbox{if $-1/n< t\leq 0$;} \\
           1, & \hbox{if $-a\leq t\leq -1/n$.}
         \end{array}
       \right.
\end{equation*}
Then $\{f_n\}\subset B_1+B_2$. Let $f\equiv 1$. It is clear that $f\in U(\{f_n\})$. Let $h\in U(\{f_n\})$. Then $\mu(f_n, h)>1/2$ for each $n$. Hence, $g(t)\geq 1$ for all $t\in [-a, -1/n]\cup [1/n, a]$ for each $n$, implying that $g(t)\geq f(t)\equiv 1$ for all $t\in [-a, a]$. Thus, $\mu(f, g)>1/2$, implying that $f\in U(g)$. This shows that $f=\sup\{f_n\}$. But $f\not\in B_1+B_2$. Therefore, $B_1+B_2$ is not a fuzzy band of $X$.

\begin{theorem}\label{theorem5.4}
Let $B_1$ and $B_2$ be two fuzzy ideals of a fuzzy Riesz space $X$. If $X=B_1\oplus B_2$, then $B_1$ and $B_2$ are fuzzy bands satisfying $B_1=B_2^d$ and $B_2=B_1^d$. In this case, we have $B_1=B_1^{dd}$ and $B_2=B_2^{dd}$.
\end{theorem}

\begin{proof}
Take $x\in B_1$ and $y\in B_2$. Since $\mu(|x|\wedge |y|, |x|)>1/2$ and $\mu(|x|\wedge |y|, |y|)>1/2$, the fuzzy solidness of $B_1$ and $B_2$ implies $|x|\wedge |y|\in B_1\cap B_2=\{0\}$. Therefore, $B_1\bot B_2$ and $B_2\subset B_1^d$.

On the other hand, take $x\in B_1^d$ such that $\mu(0, x)>1/2$. The hypothesis implies the existence of two positive elements $x_1\in B_1$ and $x_2\in B_2$ such that $x=x_1+x_2$. Since $\mu(0, x_2)>1/2$ $\mu(x_1, x_1)=1>1/2$, we have $\mu(x_1, x)>1/2$. Since $B_1^d$ is a fuzzy ideal, $x_1\in B_1^d\subset B_2$. Hence, $x_1\in B_1\cap B_2=\{0\}$. It follows that $x=x_2\in B_2$, showing that $B_1^d\subset B_2$. Therefore, $B_2=B_1^d$. By symmetry, we also have $B_2^d=B_1$.

The second statement follows from Theorem \ref{theorem4.4}.
\end{proof}

\begin{lemma}\label{lemma5.0}
Let $D$ be a nonempty subset of a fuzzy Riesz space $X$. Then $D^d=I_D^d=B_D^d$, where $I_D$ and $B_D$ are
the fuzzy ideal and fuzzy band generated by $D$, respectively.
\end{lemma}

\begin{proof}
It suffices to show that $D^d=B_D^d$. Since $D\subset B_D$, we have $B_D^d\subset D^d$. For the converse, take $x\in D^d$. Then $x\bot y$ for all $y\in D$. By Theorem \ref{theorem2.3.5} and Theorem \ref{theorem5.0}, we have $x\in B_D^d$, implying that $D^d\subset B_D^d$. Therefore, $D^d=B_D^d$.
\end{proof}

\begin{theorem}\label{theorem5.5}
Let $X$ be a fuzzy Riesz space. Then $X$ is fuzzy Archimedean if and only if $B=B^{dd}$ for all fuzzy band $B$ of $X$. \end{theorem}

\begin{proof}
Assume that $X$ is a fuzzy Archimedean Riesz space and $B$ is a fuzzy band of $X$. By Theorem \ref{theorem4.4}, we have $B\subset B^{dd}$. Thus, to show that $B=B^{dd}$, it suffices to show that $B^{dd}\subset B$. To this end, take $x\in B^{dd}$ and put
\begin{equation*}
D_x=\{y\in B^+\mid y\neq 0, y\neq x, \mu(y, x)>1/2\}.
\end{equation*}
Clearly, $D_x\neq \phi$, $D_x\uparrow$ and $x\in U(D_x)$. We show that $D_x\uparrow x$, that is, $x=\sup D_x$. Assume $x\neq \sup D_x$. Then there exists some $z\in X^+$ such that $z\in U(D_x)$ but $z\not \in U(x)$, i.e.,
$\mu(y, z)>1/2$ for all $y\in D_x$ and $\mu(z, x)>1/2$. Since $x\neq z$, $x-z\in B^{dd}$ and $B^d\cap B^{dd}=\{0\}$, we have $x-z\not \in B^{d}$. This implies that there exists some $w\in B^+$ such that
$v=w\wedge (x-z)\neq 0$. As $\mu(v, w)>1/2$ and $w\in B$, the solidness of $B$ implies $v\in B$. In view of $\mu(v, x-z)>1/2$ and $\mu(x-z, x)>1/2$, we have $\mu(v, x)>1/2$. Evidently, $0\neq v\in B^+$; hence $v\in D_x$. Thus, $\mu(v, z)>1/2$. It follows that $0\neq 2v\in D_x$ and $\mu(2v, x)=\mu(v+v, (x-z)+z)>1/2$. By induction on $n$, we have
$0\neq nv\in D_x$ and $\mu(nv, x)>1/2$, that is, the sequence $\{nv\}$ is bounded above, contradicting Theorem \ref{theorem2.3.7}. This proves that $x=\sup D_x$. As $D_x\subset B$ and $B$ is a fuzzy band, we have $x\in B$. Hence, $B^{dd}\subset B$.

Conversely, we assume that $B=B^{dd}$. Suppose $X$ is not fuzzy Archimedean. Then there exists nonnegative elements $x, y\in X$ such that $y\in U(\{nx\}_{n\in N})$. Let $I_x$ be the fuzzy ideal generated by $x$ in $X$ and put $I=I_x\oplus I_x^d$. If $z\bot I$, then $z\bot I_x$ and $z\bot I_x^d$, showing that $z\in I_x\cap I_x^d=\{0\}$. Thus, $I^d=\{0\}$, implying $X=I^{dd}$. It follows from Lemma \ref{lemma5.0} and the hypothesis that $X=B_I$, where $B_I$ is the fuzzy band generated by $I$ in $X$. Thus, $y\in B_I$ and $y=\sup D_y$, where
\begin{equation*}
D_y=\{z\mid z\in I^+, \mu(z, y)>1/2\}.
\end{equation*}
Next, let $z\in D_y$. Then Theorem \ref{theorem4.3} shows that $z=z_1+z_2$, where $z_1\in I^+_x$, $z_2\in (I_x^d)^+$, $\mu(z_1, z)>1/2$ and $\mu(z_2, z)>1/2$. In view of Theorem \ref{theorem4.2}, there exists some $k\in N$ such that $\mu(z_1, kx)>1/2$. Hence, $\mu(z_1+x, (k+1)x)>1/2$, showing that $z_1+x\in I_x$.

Therefore, $(z_1+x)\bot z_2$. Moreover, we have $\mu(z_2, z)>1/2$ and $\mu(z, y)>1/2$, implying $\mu(z_2, y)>1/2$. By Theorem \ref{theorem2.3*}, we have
$\mu(z+x, y)=\mu((z_1+x)\vee z_2, y)>1/2$, or equivalently, $\mu(z, y-x)>1/2$ for all $z\in D_y$. Thus, $y-x\in U(x)$. On the other hand, $x$ is a nonnegative element; hence $\mu(y-x, y)>1/2$, i.e., $y\in U(y-x)$. This contradicts the fact that $y=\sup D_y$. By way of contraposition, we conclude that $X$ must be fuzzy Archimedean.
\end{proof}

\begin{definition}\label{definition5.5}
\emph{Let $X$ be a fuzzy Riesz space.
\begin{enumerate}
  \item [(i)]$X$ is said to be \emph{fuzzy Dedekind complete} or \emph{fuzzy order complete} if every nonempty subset of $X$ that is bounded  above has a supremum. In this case, we also say $X$ is a \emph{fuzzy Dedekind complete Riesz space}.
  \item [(ii)]$X$ is said to be \emph{fuzzy Dedekind $\sigma$-complete} if every nonempty countable subset of $X$ that is bounded above or bounded  below has a supremum or infimum, respectively. In this case, we also say $X$ is a \emph{fuzzy $\sigma$-Dedekind complete Riesz space}.
\end{enumerate}}
\end{definition}

\begin{lemma}\label{lemma5.1}
If $X$ is a fuzzy Dedekind complete Riesz space, then $X$ is fuzzy Archimedean.
\end{lemma}

\begin{proof}
Let $x\in X^+$. In view of Theorem \ref{theorem2.3.7}, we need to show that the sequence $\{nx\}_{n\in N}$ is not bounded above. To proceed by way of contraposition, we assume that there exists some $y\in X$ such that $\mu(nx, y)>1/2$ for all $n\in N$. Since $X$ is fuzzy Dedekind complete, $x_0=\sup\{nx\}_{n\in N}$ exists; similarly, $2x_0=\sup\{2nx\}$ exists. Since $\mu(nx, 2nx)>1/2$ and $\mu(2nx, (2n+1)x)>1/2$ for all $n\in N$, we see that $\sup\{nx\}_{n\in N}=\sup\{2nx\}_{n\in N}$. Thus, we have $x_0=2x_0$, implying $x_0=0$. This further implies that $x=0$, contradicting the fact that $x$ is an arbitrary element in $X^+$. Thus, the theorem is established.
\end{proof}

\begin{theorem}\label{theorem5.6}
Let $X$ be a fuzzy Riesz space.
\begin{enumerate}
  \item [(i)]$X$ is fuzzy Dedekind complete if and only if every nonempty subset of $X^+$ that is directed to the right and bounded above has a supremum.
  \item [(ii)]$X$ is fuzzy Dedekind $\sigma$-complete if and only if every increasing sequence in $X^+$ that is bounded above has a supremum.
\end{enumerate}
\end{theorem}

\begin{proof}
\begin{enumerate}
  \item [(i)]Suppose $X$ is a fuzzy Dedekind complete, then the stated property obviously holds.
             Conversely, assume the stated property holds. Let $D$ be a nonempty subset of $X$ such that $U(D)\neq \phi$. We will show
             that $\sup D$ exists in $X$. To this end, take $x\in D$ and put
             \begin{equation*}
             E=\{x\vee y\mid y\in D\}.
             \end{equation*}
             Then $E$ is clearly directed to the right and $U(D)=D(E)\neq \phi$. Thus, it suffices to show that
             $\sup E$ exists. Define
             \begin{equation*}
             F=\{z-x\mid z\in E\}.
             \end{equation*}
             Then $F$ is still directed to the right and $U(F)\neq \phi$. Also, it is clear that $F\subset X^+$. By the hypothesis,
             $\sup F$ exists, implying that $\sup E=\sup F+x$ exists. Thus, $X$ is fuzzy Dedekind complete.
  \item [(ii)]Similar to the proof of (i).
\end{enumerate}
\end{proof}

Example 5.3 shows that the sum of two fuzzy bands of a fuzzy Riesz space need not be a fuzzy band.
However, the next theorem shows that the sum of two fuzzy bands of a fuzzy Dedekind complete  Riesz space is a fuzzy band.

\begin{lemma}\label{lemma5.3}
Let $B_1$ and $B_2$ be two disjoint fuzzy bands in a fuzzy Dedekind complete Riesz space $X$. Then $B_1\oplus B_2$ is a fuzzy band of $X$.
\end{lemma}

\begin{proof}
Let $D$ be a nonempty subset of $(B_1\oplus B_2)^+$ such that $D$ is directed to the right and $x_0=\sup D$ exists. In view of Theorem \ref{theorem5.6}, it suffices to show that $x_0\in B_1\oplus B_2$. To this end, take $x\in D$. By Theorem \ref{theorem4.3}, $x$ can be uniquely written as $x=y_x+z_x$, where $y_x\in B_1$ and $z_x\in B_2$. Notice that if $x_1, x_2, x_3\in D$ such that $\mu(x_1, x_3)>1/2$ and $\mu(x_2, x_3)>1/2$, then $\mu(y_{x_1}\vee y_{x_2}, y_{x_3})>1/2$ and $\mu(z_{x_1}\vee z_{x_2}, z_{x_3})>1/2$. Therefore, $\mu(y_x, x_0)>1/2$ and $\mu(z_x, x_0)>1/2$ for all $x\in D$. Since $X$ is fuzzy Dedekind complete, there exist $w_1, w_2\in X$ such that $w_1=\sup \{y_x\}_{x\in D}$ and $w_2=\sup \{z_x\}_{x\in D}$. As $\{y_x\}_{x\in D}\subset B_1$ and $\{z_x\}_{x\in D}\subset B_2$, we have $w_1\in B_1$ and $w_2\in B_2$. It follows that
\begin{equation*}
x_0=\sup\{y_x+z_x\}_{x\in D}=\sup \{y_x\}_{x\in D}+\sup \{z_x\}_{x\in D}=w_1+w_2\in B_1\oplus B_2.
\end{equation*}
This proves that $B_1\oplus B_2$ is a fuzzy band in $X$.
\end{proof}

\begin{theorem}\label{theorem5.7}
Let $B$ be a fuzzy band of a fuzzy Riesz space $X$. If $X$ is fuzzy Dedekind complete, then $X=B\oplus B^d$.
\end{theorem}

\begin{proof}
By Lemma \ref{lemma5.3}, $B\oplus B^d$ is a fuzzy band of $X$. Let $x\in (B\oplus B^d)^d$, then $x\in B\cap B^d=\{0\}$. Therefore, $(B\oplus B^d)^d=\{0\}$. It follows
that $B\oplus B^d=(B\oplus B^d)^{dd}=X$.
\end{proof}

\section{Fuzzy band projections}

\begin{definition}\label{definition6.1}
\emph{A fuzzy band $B$ of a fuzzy Riesz space $X$ is called a \emph{fuzzy projection band} if $X=B\oplus B^d$. }
\end{definition}

\begin{definition}\label{definition6.2}
\emph{Let $X$ be a fuzzy Riesz space. An element $x\in X$ is said to be a \emph{fuzzy projection vector} if the band generated by $x$ is a fuzzy projection band. $X$ is said to have the \emph{fuzzy principal projection property} if each element in $X$ is a fuzzy projection vector.}
\end{definition}

Let $B$ is a fuzzy projection band on a fuzzy Riesz space $X$. Then each $x\in X$ has a unique decomposition $x=x_1+x_2$, where $x_1\in B$ and $x_2\in B^d$. Therefore, the mapping $P_B: X\rightarrow X$ defined by
\begin{equation}\label{6.1}
P_B(x)=x_1
\end{equation}
is a projection.

\begin{definition}\label{definition6.3}
\emph{Let $B$ be a fuzzy projection band of a fuzzy Riesz space $X$. The projection $P_B$ defined by Equation (\ref{6.1}) is called a \emph{fuzzy band projection} on $X$. In particular, if $x$ is a projection vector of $X$, we will write $P_x$ for $P_{B_x}$.}
\end{definition}

\begin{definition}\label{definition6.4}
\emph{Let $X$ be a fuzzy Riesz space and $x, y\in X$ such that $\mu(x, y)> 1/2$. Then the set $\{z\in X\mid \mu(x, z)>1/2 \text{ and $\mu(z, y)>1/2$} \}$ is called a \emph{fuzzy order interval} and is denoted by $[x, y]$.}
\end{definition}

\begin{theorem}\label{theorem6.1}
Let $B$ be a fuzzy ideal of a fuzzy Riesz space $X$. Then the following statements are equivalent.
\begin{enumerate}
  \item [(i)]$B$ is fuzzy projection band.
  \item [(ii)]For each $x\in X^+$, the supremum of the set $D_x=\{y\in B^+\mid \mu(y, x)>1/2\}=B^+\cap [0, x]$ exists in $X$ and belongs to $B$.
  \item [(iii)]There exists a fuzzy ideal $I$ of $X$ such that $X=I\oplus B$.
\end{enumerate}
\end{theorem}

\begin{proof}
\begin{enumerate}
  \item []$(i)\Longrightarrow (ii)$ Let $x\in X^+$. By the hypothesis, $x$ can be uniquely written as
            $x=x_1+x_2$, where $x_1\in B^+$ and $x_2\in (B^d)^+$. Take any element $z\in D_x$, i.e., $z\in B^+$ such that $\mu(z, x)=\mu(z, x_1+x_2)>1/2$. Then $\mu(z-x_1, x_2)>1/2$, implying $\mu((z-x_1)^+, x_2)>1/2$.
            It follows from the fuzzy solidness of $B^d$ that $(z-x_1)^+\in B^d$. Since $(z-x_1)^+\in B$ and $B\cap B^d=\{0\}$, we have $(z-x_1)^+=0$. Hence, $\mu(z, x_1)>1/2$ for all $z\in D_x$, that is, $x_1\in U(D_x)$. As $x_1\in D_x$, we have $\sup D_x=x_1\in B$.
  \item []$(ii)\Longrightarrow (iii)$ Let $x\in X^+$. By (ii), $y=\sup D_x$ exists and belongs to $B$.
            Put $z=x-y$. Then $\mu(0, z)>1/2$. Take $w\in B^+$. We have $z\wedge w\in B^+$; hence $y+z\wedge w\in B$. It follows from Theorem \ref{theorem2.3.2*} that
            \begin{equation*}
            y+z\wedge w=(y+z)\wedge (y\wedge w)=x\wedge (y\wedge w),
            \end{equation*}
            implying that $\mu(y+z\wedge w, x)>1/2$, which further implies $\mu(y+z\wedge w, y)>1/2$.
            On the other hand, we have $\mu(0, z\wedge w)>1/2$; hence $\mu(y, y+z\wedge w)>1/2$. Therefore,
            the antisymmetry of $\mu$ implies $y=y+z\wedge w$, i.e., $z\wedge w=0$. It follows that $z\in B^d$. This proves $X=B\oplus B^d$. Finally, (iii) follows by taking $A=B^d$.
  \item []$(iii)\Longrightarrow (i)$ The conclusion follows readily from Theorem \ref{theorem5.4}.
\end{enumerate}
\end{proof}

Let $X$ and $Y$ be two fuzzy Riesz spaces and $\mu$ and $\nu$ are the associated fuzzy orders on $X$ and $Y$, respectively. Suppose $\mathcal{T}$ denotes the collection of all operators from $X$ to $Y$, that is, $\mathcal{T}=\{T\mid T: X\rightarrow Y\}$. We may equip $T$ with a partial order $\preceq$ defined by $S\preceq T$ if and only if $\nu(S(x), T(x))>1/2$ for all $x\in X$. Also, we may write $T(x)$ as $Tx$ when no confusion will arise.

\begin{definition}\label{definition6.5}
\emph{Let $X$ and $Y$ be two fuzzy Riesz spaces and $\mu$ and $\nu$ are the associated fuzzy orders on $X$ and $Y$, respectively. An operator $T: X\rightarrow Y$ is said to be \emph{fuzzy positive} if $\mu(0, x)>1/2$ implies $\nu(0, T(x))>1/2$.}
\end{definition}
\noindent \textbf{Remark.} Our definition is slightly more general than Definition 2.1 in \cite{Beg2}. Indeed, it is easy to see that if $T$ is fuzzy positive in the sense of Definition 2.1 in \cite{Beg2}, then it must be fuzzy positive in the sense of Definition \ref{definition6.3}. The next example shows that the converse need not hold.\\

\noindent \textbf{Example 6.1.} Consider $X=Y=R$. Equip $X$ with a membership function $\mu: X\times X\rightarrow [0, 1]$ defined by
\begin{equation*}
\mu(x, y)=\left\{
            \begin{array}{ll}
              1, & \hbox{ if $x=y$;} \\
              4/5, & \hbox{if $x<y$;} \\
              0, & \hbox{otherwise.}
            \end{array}
          \right.
\end{equation*}
Equip $Y$ with a membership function $v: X\times X\rightarrow [0, 1]$ defined by
\begin{equation*}
\nu(x, y)=\left\{
            \begin{array}{ll}
              1, & \hbox{ if $x=y$;} \\
              2/3, & \hbox{if $x<y$;} \\
              0, & \hbox{otherwise.}
            \end{array}
          \right.
\end{equation*}
It is easy to see $\mu$ and $\nu$ are fuzzy orders on $X$ and $Y$, respectively.
Let $T: X\rightarrow Y$ be the identity mapping, that is, $T(x)=x$ for all $x\in X$. Then $T$ is evidently fuzzy positive in the sense of Definition \ref{definition6.3}. However, $\mu(1, 2)=4/5$ and $\nu(T(1), T(2))=\nu(1, 2)=2/3\not \geq \mu(1, 2)$.
Therefore, $T$ is not fuzzy positive in the sense of Definition 2.1 in \cite{Beg2}.

\begin{theorem}\label{theorem6.2}
Let $P_B$ be a fuzzy band projection on a fuzzy Riesz space $X$. Then the following two statements hold.
\begin{enumerate}
  \item [(i)]$P_B$ is fuzzy positive.
  \item [(ii)]$P_B(x)=\sup\{y\in B^+\mid \mu(y, x)>1/2\}=\sup (B\cap [0, x])$.
\end{enumerate}
\end{theorem}

\begin{proof}
\begin{enumerate}
  \item [(i)]Let $B$ be the fuzzy projection band associated with $P_B$. Then $B\oplus B^d=X$.
             If $x\in X^+$, then the Riesz Decomposition Theorem implies $x=x_1+x_2$, where
             $x_1\in B^+$ and $X_2\in (B^d)^+$. Thus, $\mu(0, P_B(x))=\mu(0, x_1)>1/2$, showing that $P_B$ is fuzzy positive.
  \item [(ii)]By Theorem \ref{theorem6.1}, $z=\sup\{y\in B^+\mid \mu(y, x)>1/2\}$ exists and belongs to $B$.
              Let $B$ be the fuzzy band associated with $P_B$. Then each $x\in X^+$ may be uniquely written as              $x=x_1+x_2$, where $x_1\in B^+$ and $x_2\in (B^d)^+$. It is clear that $\mu(x_1, z)>1/2$   and $\mu(z, x)>1/2$ which implies $\mu(0, z-x_1)>1/2$ and $\mu(z-x_1, x-x_1)=\mu(z-x_1, x_2)>1/2$, respectively. Thus, $z-x_1\in B^d$. On the other hand, $z-x_1\in B$. Hence, $z-x_1\in B\cap B^d=\{0\}$. It follows that $z=x_1$, that is, $P_B(x)=z=\sup D_x$, proving the theorem.
\end{enumerate}
\end{proof}

Next, we apply Theorem \ref{theorem6.1} and Theorem \ref{theorem6.2} to give a characterization theorem of fuzzy projection vectors.

\begin{theorem}\label{theorem6.4}
An element $x$ in a fuzzy Riesz space $X$ is a fuzzy projection vector if and only if the supremum of the set $E_y=\{y\wedge n|x|\}_{n\in N}$ exists in $X$ for each positive element $y\in X$. In this case, we have
\begin{equation*}
P_x(y)=\sup E_y=\sup \{y\wedge n|x|\}_{n\in N}, \quad \text{for all $y\in X^+$}.
\end{equation*}
\end{theorem}

\begin{proof}
Let $x\in X$ and $y\in X^+$. Then $E_y\subset  B_x\cap [0, y]$, where $B_x$ is the fuzzy principal band generated by $x$. Thus, $U(B_x\cap [0, y])\subset U(E_y)$. Conversely, if $z\in U(E_y)$, then $\mu(y\wedge n|x|, z)>1/2$ for all $n\in N$. Take any element $w\in B_x\cap [0, y]$. Then $\mu(w\wedge n|x|, y\wedge n|x|)>1/2$,
implying $\mu(w\wedge n|x|, z)>1/2$. Since Corollary \ref{corollary5.1} shows that $w\wedge n|x| \uparrow w$, we have $\mu(w, z)>1/2$, that is, $z\in U(B_x \cap [0, y])$, implying  $U(E_y)\subset U(B_x\cap [0, y])$. Therefore, $U(E_y)= U(B_x\cap [0, y])$. Now the theorem follows from Theorem \ref{theorem6.1} and Theorem \ref{theorem6.2}.
\end{proof}

\begin{lemma}\label{lemma6.1}
Let $X$ and $Y$ be two fuzzy Riesz spaces and $\mu$ and $\nu$ be the associated fuzzy orders, respectively. If $T: X\rightarrow Y$ is a fuzzy positive operator from $X$ to $Y$, then
\begin{equation*}
\nu(|Tx|, T|x|)>1/2\quad \text{for all $x\in X$.}
\end{equation*}
\end{lemma}

\begin{proof}
Let $x\in X$. Then $\mu(x, |x|)>1/2$ and $\mu(-x, |x|)>1/2$. Thus, $\mu(0, |x|-x)>1/2$ and $\mu(0, |x|+x)1/2$. Since $T$ is fuzzy positive, we have $\mu(0, T|x|-Tx)>1/2$ and $\mu(0, T|x|+Tx)1/2$. It follows that $\mu(Tx, T|x|)>1/2$ and $\mu(-Tx, T|x|)>1/2$, respectively. This shows that $\nu(|Tx|, T|x|)>1/2$.
\end{proof}

\begin{theorem}\label{theorem6.3}
Let $X$ be a fuzzy Riesz space, $T: X\rightarrow X$ be an operator on $X$ and $I$ be the identity operator on $X$. Then the following statements are equivalent.
\begin{enumerate}
  \item [(i)]$T$ is a fuzzy band projection.
  \item [(ii)]$T$ is a fuzzy positive projection satisfying $T\preceq I$.
  \item [(iii)]$Tx\bot (I-T)y$ for all $x, y\in X$, that is, $T$ and $I-T$ have disjoint ranges.
\end{enumerate}
\end{theorem}

\begin{proof}
\begin{enumerate}
  \item []$(i) \Longrightarrow (ii)$ This is trivial.
  \item []$(ii) \Longrightarrow (iii)$ Let $x, y\in X^+$. It follows from $\mu(0, Tx\wedge (I-T)y)>1/2$,
         $\nu(Tx\wedge (I-T)y, (I-T)y)>1/2$ and the fuzzy positivity of $T$ and $I-T$ that $\mu(0, T(Tx\wedge (I-T)y))>1/2$ and $\nu(T(Tx\wedge (I-T)y), 0)=\nu(T(Tx\wedge (I-T)y), T((I-T)y))>1/2$.
         Therefore, the antisymmetry of $\nu$ implies $T(Tx\wedge (I-T)y)=0$. Similarly, $(I-T)(Tx\wedge (I-T)y)=0$. Hence, $Tx\wedge (I-T)y=T(Tx\wedge (I-T)y)+(I-T)(Tx\wedge (I-T)y)=0$.
         In view of Lemma \ref{lemma6.1}, we conclude that  $Tx\bot (I-T)y$ for all $x, y\in X$.
  \item []$(iii) \Longrightarrow (i)$ Suppose $B_1$ and $B_2$ are the fuzzy ideals generated by the ranges of $T$ and $I-T$, respectively. Then $B_1\bot B_2$. For every $x\in X$, we have
            $x=Tx+(I-T)x$. Therefore, $X=B_1\oplus B_2$. It follows from Theorem \ref{theorem5.4} that $B_1$ and $B_2$ are fuzzy projection bands on $X$. Thus, we have
                \begin{equation*}
                P_{B_1}(x)-T(x)=P_{B_1}(x)-P_{B_1}T(x)=P_{B_1}(I-T)(x)=0,
                \end{equation*}
                showing that $T=P_{B_1}$. Hence, $T$ is a fuzzy band projection.
\end{enumerate}
\end{proof}

\begin{corollary}\label{corollary6.1}
If $X$ is a fuzzy Dedekind $\sigma$-complete Riesz space, then $X$ has the fuzzy principal projection property. \end{corollary}

\begin{theorem}\label{theorem6.5}
If $B_1$ and $B_2$ are two fuzzy projection bands of a fuzzy Riesz space $X$, then $B_1^d, B_1\cap B_2$ and $B_1+B_2$ are also fuzzy projection bands satisfying the following identities.
\begin{enumerate}
  \item [(i)]$P_{B_1^d}=I-P_{B_1}$.
  \item [(ii)]$P_{B_1\cap B_2}=P_{B_1}P_{B_2}=P_{B_2}P_{B_1}$.
  \item [(iii)]$P_{B_1+B_2}=P_{B_1}+P_{B_2}-P_{B_1}P_{B_2}=P_{B_1}+P_{B_2}-P_{B_1\cap B_2}$.
\end{enumerate}
\end{theorem}

\begin{proof}
\begin{enumerate}
  \item [(i)]Since $B_1$ is a fuzzy projection band, we have $X=B_1\oplus B_1^d$. By Theorem \ref{theorem5.4}, we know $X=B_1^d\oplus B_1^{dd}$, that is,
            $B_1^d$ is a fuzzy projection band. It is obvious that $P_{B_1}+P_{B_1^d}=I$, that is, $P_{B_1^d}=I-P_{B_1}$.
  \item [(ii)]Let $x\in X^+$. Apply Theorem \ref{theorem6.2} to $P_{B_2}$ to obtain $B_1\cap [0, P_{B_2}(x)]=B_1\cap B_2\cap [0, x]$. Then apply Theorem \ref{theorem6.2}
                to $P_{B_1}$ to get
             \begin{equation*}
             P_{B_1}P_{B_2}x=\sup (B_1\cap [0, P_{B_1}(x)])=\sup (B_1\cap B_2\cap [0, x]).
             \end{equation*}
             It follows from Theorem \ref{theorem6.1} that $B_1\cap B_2$ is a fuzzy projection band and $P_{B_1\cap B_2}=P_{B_1} P_{B_2}$. By symmetry, we also have
             $P_{B_1\cap B_2}=P_{B_2} P_{B_1}$.
  \item [(iii)]First, we assume that $B_1\bot B_2$. Let $x\in X^+$ and take $x_1+x_2\in (B_1+B_2)\cap [0, x]$. Then $x_1\in B_1\cap [0, x]$, $x_2\in B_2\cap [0, x]$,
               $x_1+x_2\in (B_1+B_2)^+$, and $\mu(x_1+x_2, x)>1/2$. Thus, Theorem \ref{theorem6.2} implies $\mu(x_1, P_{B_1}(x))>1/2$ and $\mu(x_2, P_{B_2}(x))>1/2$, which further implies $\mu(x_1+x_2, P_{B_1}(x)+P_{B_2}(x))>1/2$. Thus, $P_{B_1}(x)+P_{B_2}(x)\in U((B_1+B_2)\cap [0, x])$. As $P_{B_1}(x)+P_{B_2}(x)\in B_1+B_2$, we have
               \begin{equation}\label{6.2}
               \sup \left((B_1+B_2)\cap [0, x]\right)=P_{B_1}x+P_{B_2}x.
               \end{equation}
               By Theorem \ref{theorem6.1}, we know $B_1+B_2$ is a fuzzy projection band. Then it follows from Equation (\ref{6.2}) and Theorem \ref{theorem6.2} that $P_{B_1+B_2}=P_{B_1}+P_{B_2}$. For the general case, we notice that $B_1+B_2=(B_1\cap B_2^d)\oplus B_2$. Thus, by (i), (ii) and the above special case, we have
               \begin{eqnarray*}
               P_{B_1+B_2} &= & P_{(B_1\cap B_2^d)\oplus B_2}=P_{B_1\cap B_2^d}+P_{B_2}\\
                           &=& P_{B_1}P_{B_2^d}+P_{B_2}\\
                           &=& P_{B_1}(I-P_{B_2})+P_{B_2}\\
                           &=& P_{B_1}+B_{B_2}-P_{B_1}P_{B_2}=P_{B_1}+B_{B_2}-P_{B_1\cap B_2}.
               \end{eqnarray*}
\end{enumerate}
\end{proof}

\begin{corollary}\label{corollary6.2}
If $x$ and $y$ are two fuzzy projection vectors in a fuzzy Riesz space $X$, then $B_x^d, B_x\cap B_y$ and $B_x+B_y$ are also fuzzy projection bands satisfying the following identities.
\begin{enumerate}
  \item [(i)]$P_{B_x^d}=I-P_{x}$.
  \item [(ii)]$P_{x\wedge y}=P_{x}P_{y}=P_{y}P_{x}$.
  \item [(iii)]$P_{x+y}=P_{x}+P_{y}-P_{x}P_{y}=P_{x}+P_{y}-P_{x\wedge y}$.
\end{enumerate}
\end{corollary}

\begin{theorem}\label{theorem6.6}
Let $B_1$ and $B_2$ be two fuzzy projection bands on a fuzzy Riesz space $X$. Then the following statements are equivalent.
\begin{enumerate}
  \item [(i)]$B_1\subset B_2$.
  \item [(ii)]$P_{B_1}P_{B_2}=P_{B_2}P_{B_1}=P_{B_1}$.
  \item [(iii)]$P_{B_1}\preceq P_{B_2}$.
\end{enumerate}
\end{theorem}

\begin{proof}
\begin{enumerate}
  \item []$(i) \Longrightarrow (ii)$ If $B_1\subset B_2$, then Theorem \ref{theorem6.5} implies that
            \begin{equation*}
            P_{B_1}P_{B_2}=P_{B_2}P_{B_1}=P_{B_1\cap B_2}=P_{B_1}.
            \end{equation*}
  \item []$(ii) \Longrightarrow (iii)$ Suppose (ii) holds. Then for each positive element $x\in X$ Theorem \ref{theorem6.3} implies
             \begin{equation*}
             P_{B_1}(x)=P_{B_1}P_{B_2}(x)\preceq IP_{B_2}(x)=P_{B_2}(x),
             \end{equation*}
             showing that $P_{B_1}\preceq P_{B_2}$.
  \item []$(iii) \Longrightarrow (i)$ If (iii) holds, then for each positive element $x\in B_1$ we have
            \begin{equation*}
            x=P_{B_1}(x)\leq P_{B_2} (x)\in B_2,
            \end{equation*}
            implying that $B_1\subset B_2$.
\end{enumerate}
\end{proof}

\bibliographystyle{amsplain}

\end{document}